\documentclass[11pt,reqno]{amsart}
\usepackage{amssymb,amsmath,amsthm,amsfonts}
\usepackage{mathrsfs,a4wide,dsfont}
\usepackage[T1]{fontenc}
\usepackage{color}
\usepackage{comment}

\theoremstyle{plain}
\newtheorem{theorem}{Theorem}[section]
\newtheorem{lemma}[theorem]{Lemma}
\newtheorem{proposition}[theorem]{Proposition}

\newtheorem{remark}[theorem]{Remark}
\newtheorem{definition}[theorem]{Definition}
\theoremstyle{definition}
\theoremstyle{remark}
\numberwithin{equation}{section}

\newcommand{\hs}{{\mathcal  H}}

\newcommand{\R}{{\mathbb R}}

\newcommand{\N}{{\mathbb N}}

\newcommand{\smat}[1]{\left(\begin{smallmatrix} #1 \end{smallmatrix}\right)}

\newcommand{\weak}{\rightharpoonup}
\newcommand{\weakstar}{\stackrel{*}{\rightharpoonup}}

\newcommand{\Ha}{\hs}

\newcommand{\mres}{\mathbin{\vrule height 1.6ex depth 0pt width 0.13ex\vrule height 0.13ex depth 0pt width 1.6ex}}

\DeclareMathOperator{\tr}{trace}
\DeclareMathOperator{\Id}{Id}
\DeclareMathOperator{\cof}{cof}
\DeclareMathOperator{\jac}{J}
\DeclareMathOperator{\spann}{span}
\DeclareMathOperator{\spt}{spt}

\newcommand{\K}{{\mathcal K}}

\newcommand{\e}{\varepsilon}
\newcommand{\eps}{\varepsilon}

\newcommand{\Msp}{\mathcal M} 
\newcommand{\Xsp}{\mathcal X} 
\newcommand{\Qfu}{\mathcal Q} 
\newcommand{\evla}{\lambda_1} 
\newcommand{\evmu}{\lambda_2} 
\newcommand{\pGG}{\Phi} 
\newcommand{\perm}{\varepsilon}

\newcommand{\vf}[1]{\mathcal{V}(#1)}
\newcommand{\ovf}[1]{\mathcal{V}^o(#1)}

    \let\TeXchi\chi
\newbox\chibox
\setbox0 \hbox{\mathsurround0pt $\TeXchi$}
\setbox\chibox \hbox{\raise\dp0 \box 0 }
\def\chi{\copy\chibox}

\title
[Gamma limit for a bending energy with tilt]
{Gamma convergence of a family of surface--director bending energies with small tilt}
\author[L. Lussardi]
{Luca Lussardi}
\address[L.\,Lussardi]{Dipartimento di Matematica e Fisica ``N.\,Tartaglia'', Universit\`a Cattolica del Sacro Cuore, via dei Musei 41, I-25121 Brescia, Italy}
\email[]{luca.lussardi@unicatt.it}
\urladdr{http://www.dmf.unicatt.it/~lussardi/}
\author[M.\,R\"oger]
{Matthias R\"oger}
\address[M.\,R\"oger]{Fakult\"at f\"ur Mathematik, Technische Universit\"at Dortmund, Vogelpothsweg 87, D-44227 Dortmund, Germany}
\email[]{matthias.roeger@math.tu-dortmund.de}
\urladdr{http://www.mathematik.tu-dortmund.de/lsxi/roeger/}

\begin{document}

\begin{abstract}
\small{We prove a Gamma-convergence result for a family of bending energies defined on smooth surfaces in $\R^3$ equipped with a director field. The energies strongly penalize the deviation of the director from the surface unit normal and control the derivatives of the director. Such type of energies for example arise in a model for bilayer membranes introduced by Peletier and R\"oger [Arch. Ration. Mech. Anal. 193 (2009)]. Here we prove in three space dimensions in the vanishing-tilt limit a Gamma-liminf estimate with respect to a specific curvature energy. In order to obtain appropriate compactness and lower semi-continuity properties we use tools from geometric measure theory, in particular the concept of generalized Gauss graphs and curvature varifolds. 
}
\end{abstract}

\maketitle
\vskip .2truecm

\section{Introduction}

%
Curvature functionals arise in many applications from physics and biology and have been intensively studied over the past decades. In the modeling of biomembranes a prominent example are shape energies of Canham--Helfrich type \cite{Canh70,Helf73}. These are of the general form
\begin{align}
  \mathcal{E}(S) \,&=\, \int_S k_1 (H-H_0)^2\,d\Ha^{2}\,
  +\,\,\int_S  {k}_2 K\,d\Ha^{2}, \label{eq:E-sharp}
\end{align}
where $S$ denotes a surface in $\R^3$,
$H$ and $K$ its mean and Gaussian curvature, and where the bending moduli $k_1,k_2$ and the spontaneous curvature $H_0$ are constant. In the simplest case of zero spontaneous curvature and for fixed topological type the functionals basically reduce to the Willmore functional, that has attracted a lot of attention \cite{Simo93,KuSc12,Rivi07}.

Several refined models and variational approaches to derive such bending energies have been recently investigated, see for example \cite{Merl13,Merl13a,PaTr13,SeFr14}. In \cite{PeRoe09} a meso-scale model for biomembranes has been introduced and has been shown to converge in the macro-scale limit in two dimensions to a generalized elastica functional. 
Together with M.~A.~Peletier we have addressed the three dimensional case \cite{LuPR14} and have proved a general lower bound for the approximate functionals. Moreover we have (formally) identified the Gamma limit and have provided a corresponding \emph{limsup} construction. 

In this paper we study the asymptotic behavior of a closely related family of functionals and prove a compactness and \emph{liminf} statement. The functionals are defined on compact orientable surfaces $S\subset \R^3$ given as boundary of an open set in $\R^3$ and equipped with a Lipschitz continuous unit-vector field $\theta:S\to S^2$ satisfying $\theta\cdot \nu>0$, where $\nu:S\to S^2$ denotes the outer unit normal field of $S$. For such pairs we consider 
\begin{align}
	\Qfu_\e(S,\theta) \,:=\frac{1}{\e^2}\int_S  \Big(\frac{1}{\theta\cdot\nu}-1\Big)\,d\hs^2 + \int_S Q(L(p)) \,d\hs^{2}(p) \label{eq:def-Qfu-intro}
\end{align}
where the linear map $L(p):\R^3\to\R^3$ denotes the extension of $D_S\theta$ by $L\theta =0$, and where the  quadratic form $Q$ is defined for an arbitrary square matrix $A\in \R^{3\times 3}$ by 
$$
	Q(A):=\frac{1}{4}(\tr  A)^2-\frac{1}{6}\tr \cof A
$$
with $\cof A$ denoting the cofactor matrix of $A$. Note that the first term in $\Qfu_\e$ penalizes the deviation of $\theta$ from the unit normal whereas the second term in the case of $\nu\equiv \theta$ reduces to the curvature functional 
\begin{align*}
	\Qfu_0(S) \,:= \int_S \bigg(\frac{1}{4}H^2-\frac{1}{6}K\bigg)\,d\mathcal H^2,
\end{align*}
see Lemma \ref{lem:evDtheta} below.

The particular form $\Qfu_\e$ arises from \cite{LuPR14}, but can also be seen as a specific example of a more general class of functionals that are not only determined by the surface and its unit normal, but also depend on a director field and its deviation from the normal direction. This situation appears quite natural, see for example the discussion in Section 4 of \cite{SeFr14} or the membrane energy in \cite{HuZE07}. We expect that our strategy to prove the variational convergence for the particular functionals $\Qfu_\e$ applies to a large class of similar models.

Letting $\eps\to 0$ the functional $\Qfu_0$ is the natural candidate for the Gamma limit of $\Qfu_\e$ (with respect to convergence of the associated surface measures), at least in $C^2$-regular limit points. The corresponding \emph{limsup} estimate follows from the existence of a recovery sequences proved in \cite[Theorem 2.5]{LuPR14}. Addressing the \emph{liminf} inequality and compactness properties we face substantial difficulties: For a sequence $(S_\e,\theta_\e)$ as above, even in the `best case' that $\theta_\e\equiv \nu_\e$ we only obtain an $L^2$-bound for the second fundamental form. This however only ensures weak compactness properties in spaces of generalized surfaces (for example in the class of Hutchinson's curvature varifolds, see below). In general, the situation is much worse: if $\theta_\e$ deviates from $\nu_\e$ we do not control the second fundamental form (not even the mean curvature) of the surfaces $S_\e$. This makes any partial integration formulas for derivatives of $\theta_\e$ (typically used to characterize curvatures in the limit) useless, as non-controlled curvature terms would appear. We therefore do not pass to the limit in the sense of varifolds but use rather techniques motivated by the theory of generalized Gauss graphs as developed by Anzellotti, Serapioni and Tamanini \cite{AnST90} and further developed in particular by Delladio in a series of papers \cite{AnDe95,DeSc95,Dell96,Dell97}. For a similar strategy in a related but different problem see \cite{Mose12}. 

Let us describe our approach in more detail: We consider the graphs  $G_\e:=\{(p,\theta_\e(p)) : p\in S_\e\}$  of $\theta_\e$ over $S_\e$ and the associated currents. A bound on $\Qfu_\e(S_\e,\theta_\e)$ then implies a bound on the area of $G_\e$. Next, we expect that $\theta_\e$ becomes orthogonal to $S_\e$ when $\e$ is small (see the very definition of $\Qfu_\e$), and thus we expect that the limit $G$ of $G_\e$, in the sense of currents, is the graph of a normal to a generalized surface $S$ in $\R^3$, that is a so called generalized Gauss graphs. For such currents a theory has been developed (see \cite{AnST90}) and precisely there exists a good and stable notion of curvatures which permits us to prove the key lower bound for the limit functional. 
Therefore we rephrase the energy functional in terms of the graph associated to $(S_\e, \theta_\e)$ and prove appropriate lower semicontinuity properties. Finally, we obtain that the limit is given by a curvature varifold in the sense of Hutchinson \cite{Hutc86}, which also induces a more concise form of the generalized limit energy.

The paper is organized as follows. First of all we give a precise introduction of the problem and we state the main results in Section \ref{sec:setting}. In Section \ref{sec:Rcurr} we review some facts from differential geometry and geometric measure theory that we need, in particular regarding generalized Gauss graphs and varifolds. Then, Section \ref{sec:proof} is dedicated to the proof of the main Theorem \ref{thm:main3}.
Finally, in the appendix we provide a more detailed description of the relation of the energy \eqref{eq:def-Qfu-intro} to the mesoscale biomembrane energy analyzed in \cite{PeRoe09}, \cite{LuPR14} and recall some facts from  linear and exterior algebra.


\section{Setting of the problem and main results}\label{sec:setting}
We fix $\Omega\subset\R^3$ open.
Let $\Msp$ be the set of tuples $(S,\theta)$, where $S$ is a compact and orientable surface of class $C^2$ in $\R^3$ that is given by the boundary of an open set $A(S)\subset\subset \Omega$, and where  $\theta \colon S \to \R^3$ is a Lipschitz vector field such that
\begin{align}
	&|\theta| \,=\, 1 \text{ and }\theta \cdot \nu \,>\, 0\text{ on }S, \label{eq:cdt-theta1}\\
	& L(p)\in\R^{3\times 3} \text{ is symmetric for all }p\in S, \label{eq:cdt-theta2}
\end{align}
where $\nu:S\to\R^3$ denotes the outer unit normal field on $S$, and where $L(p):\R^3\to\R^3$ is the extension of $D\theta(p):T_p S\to \R^3$ defined by the properties
\begin{align}
	L(p)\tau \,=\, D\theta(p)\tau \quad\text{ for all }\tau\in T_pS,\qquad L(p)\theta(p) \,=\, 0. \label{eq:cdt-theta3}
\end{align}
Together with $|\theta|=1$ on $S$ this implies that 
\begin{align}
	L(p)(\R^3)\,\subset\, \theta(p)^\perp. \label{eq:theta4}
\end{align}
We next define the functional $\Qfu_\e:\Msp\to \R^+_0$, $\e>0$ by
\begin{align}
	\Qfu_\e(S,\theta) \,=\, \int_S \e^{-2} \Big(\frac{1}{\theta\cdot\nu}-1\Big)\,d\hs^2 + \int_S Q(L(p)) \,d\hs^{2}(p) \label{eq:def-Qfu}
\end{align}
for $(S,\theta)\in\Msp$, where the quadratic form $Q$ is defined for an arbitrary square matrix $A$ by 
\begin{align}
	Q(A):=\frac{1}{4}(\tr  A)^2-\frac{1}{6}\tr \cof A \label{eq:def-Q}
\end{align}
with $\cof A$ denoting the cofactor matrix of $A$.

\begin{remark}
By \cite[Lemma 3.6] {LuPR14} (see Lemma \ref{lem:evDtheta} in the Appendix) $Q$ is a positive quadratic form in the `nontrivial' eigenvalues of $D\theta$, more precisely: for any $p\in S$ such that $D\theta(p)\in \R^{3\times 3}$ exists,
\begin{align*}
	Q(D\theta(p))&=\frac{1}{4}(\evla(p)+\evmu(p))^2-\frac{1}{6}\evla(p)\evmu(p) = \frac{1}{6}(\evla(p)+\evmu(p))^2+\frac{1}{12}(\evla(p)^2+\evmu(p)^2),
\end{align*}
where $\evla(p),\evmu(p)\in\R$ are the eigenvalues of the restriction of $D\theta(p)$ to $\theta(p)^\perp$.
This shows in particular, that $Q$ controls the full matrix $D\theta(p)$ and that in the case $\theta=\nu$ we have $Q(D\nu)=\frac{1}{4}H^2-\frac{1}{6}K$ and an $L^2$-control on the second fundamental form.
\end{remark}

The main result are the following compactness and lower bound statements.

\begin{theorem}\label{thm:main3}
Let $(\e_j)_{j\in\N}$ be an infinitesimal sequence of positive numbers and $(S_j,\theta_j)_{j\in\N}$ be a sequence in $\Msp$ such that 
\begin{align}
	\sup_j \Ha^2(S_j) \,&<\,\infty, \label{eq:bd-area}\\
	\bigcup_j S_j \,&\subset\, \tilde{\Omega}\quad\text{ for some }\tilde\Omega \subset\subset \Omega, \label{eq:contain}
\end{align}
and that for a fixed $\Lambda>0$
\begin{align}
	\Qfu_{\e_j}(S_j,\theta_j)\,\leq\, \Lambda\quad\text{ for all }j\in\N.  \label{eq:Lambda}
\end{align}
Assume furthermore that in the sense of Radon measures on $\Omega$
\begin{align}
	\Ha^2\mres S_j\,\to\, \mu \quad \text{ as }j\to\infty. \label{eq:conv-Chi}
\end{align}
Then $\mu=\mu_V$ where $V$ is an integral varifold with generalized second fundamental form in $L^2$ and 
\begin{equation}\label{eq:main3}
	\int\bigg(\frac{1}{4}H^2-\frac{1}{6}K\bigg)\,dV \,\leq\, \liminf_{j\to+\infty}\Qfu_{\e_j}(S_j,\theta_j)
\end{equation}
holds, where $H$ and $K$ are, respectively, the mean curvature and the Gauss curvature of $V$ in the sense of Definition \ref{def:HutchK}. 
\end{theorem}

%
\section{Currents and generalized Gauss graphs}\label{sec:Rcurr}

Here we review some notions from differential geometry and discuss two generalizations of surfaces that we will use in the sequel: generalized Gauss graphs introduced by Anzellotti, Serapioni and Tamanini \cite{AnST90}, and curvature varifolds in the sense of Hutchinson \cite{Hutc86}.

\subsection{Differential geometry of smooth surfaces}\label{sec:surfaces}
Let $S$ be an oriented compact surface of class $C^2$, embedded in $\R^3$ and without boundary. Let $\nu \colon S \to S^2$ denote a $C^1$ unit normal field (Gauss map). The differential of the Gauss map in $p\in S$ defines a self-adjont linear map $D\nu(p) \colon T_pS\to T_pS$, 
thus $D\nu(p)$ has two real eigenvalues $\kappa_1(p),\kappa_2(p)$, the {\it principal curvatures} of $S$ in $p$. We define the mean and Gaussian curvature by
$$
H(p) \,:=\,\tr D\nu(p)\,=\, \kappa_1(p)+\kappa_2(p), \quad K(p) \,:=\, \det D\nu(p) \,=\, \kappa_1(p)\kappa_2(p),
$$
respectively. We denote by $ P(p):\R^3\,\to\, T_pS$ the orthogonal projection on the tangent space.  Extending functions $f\in C^1(S)$ to $C^1$-functions in a neighborhood of $S$ the covariant derivative is expressed by $\delta_i f\,:=\,\sum_jP_{ij}\partial_j f$, $i=1,2,3$, on $S$ and is independent of the choice of extension.

By the divergence theorem on surfaces one derives \cite[Sec.\,5.1]{Hutc86} that for all $\varphi\in C^1(\R^3\times \R^{3\times 3})$, $\varphi = \varphi(x,P)$
\begin{align}
	0 \,=\, \int_S \Big(\delta_i \varphi + \sum_{j,k}(\delta_i  P_{jk})\partial^*_{jk}\varphi + \sum_j(\delta_j  P_{ij})\varphi\Big)\,d\mathcal H^2\label{eq:Hutch-pInt}
\end{align}
holds, where $\partial^*$ denotes derivatives with respect to the $P$ variables. This relation has been used by Hutchinson \cite{Hutc86} to define a suitable notion of generalized surfaces as a class of integral varifolds with generalized second fundamental form, see Section \ref{sec:CV} below.

To give a generalized formulation of the mean and Gaussian curvature we will use the following identities that hold in the smooth case.
\begin{lemma}\label{lem:HutchK}
For a smooth surface $S$ with $C^1$ unit normal field $\nu$ let us extend $D\nu(p)\colon T_pS\to T_pS$ to a map $L(p):\R^3\to\R^3$ by setting $L(p):= D\nu(p)\circ P(p)$. Then for all $p\in S$
\begin{align}
	H(p) \,&=\,\tr L(p)\,=\, \sum_{1\leq i,j\leq 3} 
	A_{iji}(p)\nu_j(p), \label{eq:Hutch-H}\\
	K(p)\,&=\, \tr\cof L(p) \,=\, \nu(p)\cdot \cof L(p)\nu(p)\,=\, \sum_k \tr \cof (A_{ijk})_{ij} \label{eq:Hutch-K}
\end{align}
hold, where $A_{ijk}:= \delta_i  P_{jk}$.
\end{lemma}
\begin{proof}
We drop the dependence on $p$ for simplicity. To prove \eqref{eq:Hutch-H} we have, by the very definition of $L$, 
$$
\sum_{i,j}A_{iji}\nu_j=\sum_{i,j}\delta_i P_{ji}\nu_j=\sum_{i,j,h}P_{ih}\partial_hP_{ji}\nu_j=-\sum_{i,j,h}P_{ih}P_{ji}\partial_h\nu_j=\sum_{j,h}P_{jh}\partial_h\nu_j
$$
$$
=\sum_j\delta_j\nu_j=\sum_jL_{jj}=\tr L.
$$
To prove \eqref{eq:Hutch-K} first of all we notice that
$$
L_{ij}\nu_k=\delta_i\nu_j\nu_k=\delta_iP_{jk}-\delta_i\nu_k\nu_j=\delta_iP_{jk}-L_{ik}\nu_j
$$ 
that is
$$
A_{ijk}=\delta_iP_{jk}=L_{ij}\nu_k+L_{ik}\nu_j
$$
from which
$$
\begin{aligned}
\tr \cof (A_{ijk})_{ij}&= \tr \cof (L_{ij}\nu_k+L_{ik}\nu_j)_{ij}\\
&=(L_{11}\nu_k+L_{1k}\nu_1)(L_{22}\nu_k+L_{2k}\nu_2)-(L_{12}\nu_k+L_{1k}\nu_2)(L_{21}\nu_k+L_{2k}\nu_1)\\
&\quad+(L_{11}\nu_k+L_{1k}\nu_1)(L_{33}\nu_k+L_{3k}\nu_3)-(L_{13}\nu_k+L_{1k}\nu_3)(L_{31}\nu_k+L_{3k}\nu_1)\\
&\quad +(L_{22}\nu_k+L_{2k}\nu_2)(L_{33}\nu_k+L_{3k}\nu_3)-(L_{23}\nu_k+L_{2k}\nu_3)(L_{32}\nu_k+L_{3k}\nu_2).
\end{aligned}
$$
Now, by simple algebra, we obtain, since $\sum_k\nu_k^2=1$,
$$
\begin{aligned}
\sum_k\tr \cof (A_{ijk})_{ij}&=(L_{11}L_{22}-L_{12}L_{21})+(L_{11}L_{33}-L_{13}L_{31})+(L_{22}L_{33}-L_{23}L_{32})\\
&=\tr\cof L
\end{aligned}
$$
which yields the conclusion.
\end{proof}
\subsection{Rectifiable currents}
We first fix some notation from exterior algebra, see the Appendix for a more detailed exposition.\\
We denote by $\Lambda^k(\R^n),\,0\leq k\leq n$ and by $\Lambda_k(\R^n)$ the spaces of all {\it $k$-vectors} and {\it $k$-covectors}, respectively, in $\R^n$.
We call $v$ a {\it simple $2$-vector} if $v$ can be written as $v=v_1\wedge v_2$. If in addition $v\neq 0$ the space $\spann(v_1,v_2)$ is called the {\it enveloping subspace}. 
In the context of graphs it will be useful to distinguish two copies $\R^3_x$ and $\R^3_y$ of $\R^3$. The \emph{stratification} of a 2-vector $\Lambda^2(\R^3_x\otimes \R^3_y)$ is the unique decomposition
\begin{align}
	\xi=\xi_0+\xi_1+\xi_2,\quad \xi_0\,\in\, \Lambda^2(\R^3_x),\quad \xi_1\in \Lambda^1(\R^3_x)\wedge \Lambda^1(\R^3_y), \quad \xi_2\in \Lambda^2(\R^3_y) \label{eq:def-strati}
\end{align}
and is given by
\begin{align*}
	\xi_0 \,&=\, \sum_{1\leq i<j\leq 3} 
	\langle dx^i\wedge dx^j,\xi\rangle e_i\wedge e_j,\\
	\xi_1 \,&=\, \sum_{1\leq i,j\leq 3} 
	\langle dx^i\wedge dy^j,\xi\rangle e_i\wedge \eps_j,\\
	\xi_2 \,&=\, \sum_{1\leq i<j\leq 3} 
	\langle dy^i\wedge dy^j,\xi\rangle \eps_i\wedge \eps_j.
\end{align*}
where $\{e_1,e_2,e_3\}$ and $\{\varepsilon_1,\e_2,\varepsilon_3\}$ denote the standard basis for $\R^3_x$ and $\R^3_y$, respectively, and $\{dx^1,dx^2,dx^3\}$, $\{dy^1,dy^2,dy^3\}$ the corresponding dual basis. 

For $U\subseteq \R^n$ open and $k \in \{0,\dots,n\}$ we denote by $\mathcal D^k(U)$ the space of all $k$-differential forms with compact support in $U$, equipped with usual topology of distributions.

The space $\mathcal D_k(U)$ of {\it $k$-currents on $U$} is the dual of $\mathcal D^k(U)$. We denote by $\partial T\in \mathcal D_{k-1}(U)$ the  {\it boundary} of $T \in \mathcal D_k(U)$ and the {\it mass} of $T\in \mathcal D_k(U)$ in $W\subset U$ open by ${\mathbb M}_W(T)$.

Given $E \subseteq \R^n$ we say that $E$ is {\it $k$-rectifiable} if $E$ can be covered by a countable family of sets $\{S_j\}$, $j \in \N$, such that $S_0$ is $\mathcal  H^k$-negligible and $S_j$ is a $k$-dimensional surface in $\R^n$ of class $C^1$, for any $j>0$. It turns out that for $\mathcal  H^k$-almost any $p \in E$ there is a well-defined measure-theoretic tangent space $T_p E$. We say that a map $p \in E \mapsto \eta(p)$ is an {\it orientation} on $E$ if such a map is $\mathcal  H^k$-measurable and $\eta(p)$ is a unit simple $k$-vector on $\R^n$ that spans $T_pE$ for $\mathcal  H^k$-almost any point $p\in E$.  Let $\beta \colon E \to \N$ be a $\mathcal H^k$-locally summable function. Then, if $E\subset U$ with $U$ open in $\R^n$ we can define a current $T=\tau(E,\beta,\eta) \in \mathcal D_k(U)$ by
\begin{equation}\label{rect-curr}
\langle T,\omega\rangle:=\int_E \langle \omega,\eta\rangle\,\beta\,d\mathcal  H^k.
\end{equation}
The function $\beta$ is also called the {\it multiplicity} of $T$. The set $\mathcal R_k(U)$ of currents $T\in \mathcal D_k(U)$ which can be written in the form $T=\tau(E,\beta,\eta)$ as above are called {\it rectifiable currents}. 

The importance of the class of rectifiable currents stems mainly from the compactness property given by the following celebrated Federer-Fleming theorem (see for example \cite{Fede69}).
\begin{theorem}\label{comp-federer}
Let $(T_l)_{l\in\N}$ be a sequence in $\mathcal R_k(U)$ such that $\partial T_l \in \mathcal R_{k-1}(U)$ for any $l \in \N$. Assume that for any $W$ relatively compact in $U$ there exists a constant $c_W>0$ such that 
$$
{\mathbb M}_W(T_l)+{\mathbb M}_W(\partial T_l)<c_W.
$$
Then, there exist a subsequence $l_j\to\infty$ and $T \in \mathcal R_k(U)$ such that $T_{l_j}\weak T$ as $j \to +\infty$.
\end{theorem}
\subsection{Generalized Gauss graphs}
For the general theory of generalized Gauss graphs we refer the reader to \cite{AnST90}. To recall the motivation let first $S$ be a $2$-dimensional surface of class $C^2$ embedded in $\R^{3}$ and contained in an open set $\Omega\subset \R^3$, and let $\nu \colon S \to S^2$ be its Gauss map. It is convenient to distinguish the ambient space $\R^{3}_x$ of $S$ and the ambient space $\R^{3}_y$ of $\nu(S)$. Consider the graph of the Gauss map  
$$
G:=\{(p,\nu(p)) \in \R^{3}_x \times \R^{3}_y : p\in S\}
$$
Then, $G$ is a $2$-dimensional $C^1$ surface embedded in $\R^{3}_x \times \R^{3}_y$; if $S$ has boundary $\partial S$ then also $G$ has boundary given by
$$
\partial G=\{(p,\nu(p)) : p\in \partial S\}.
$$
We let $\Phi\colon S \to \R^{3}_x \times \R^{3}_y$ be given by $\Phi(p):=(p,\nu(p))$ which is of class $C^1$ on $S$. We equip $S$ with the orientation induced by $\nu$ and let $\tau(p):=\ast\nu(p)$, where 
$$
\ast \colon \Lambda^1(\R^{3}) \to \Lambda^2(\R^{3})
$$
is the Hodge operator.

Notice that in particular $\tau(p)\in\Lambda^2(T_pS)$ for any $p\in S$, thus the field $p\mapsto \tau(p)$ is a tangent $2$-vector field on $S$. We then define $\xi \colon G \to \Lambda^2(\R^{3}_x \times \R^{3}_y)$ as
$$
\xi(p,\nu(p)):=D\Phi_p(\tau_1(p))\wedge D\Phi_p(\tau_2(p)), \quad \tau=\tau_1\wedge \tau_2.
$$
It is easy to see that $|\xi|\geq 1$, and thus we can normalize $\xi$ obtaining
$$
\eta:=\frac{\xi}{|\xi|}
$$
which is an orientation on $G$.

We then can associate to $G$ the current $T_G \in \mathcal R_2(\Omega \times \R^{3}_y)$ given by $T=\tau(G,1,\eta)$. This leads to the definition of \emph{generalized Gauss graphs} as currents $T \in \mathcal R_2(\Omega \times \R^{3}_y)$ that share certain additional properties which are in particular satisfied by weak limits of Gauss graphs $T_G$ associated to graphs of $C^2$ surfaces as above (see for example \cite{AnST90}). To prepare the definition we introduce two forms $\varphi^\ast \in\mathcal D^2(\R^{3}_x \times \R^{3}_y)$ and $\varphi \in \mathcal D^1(\R^{3}_x \times \R^{3}_y)$ given by
\begin{align}
	\varphi(x,y) \,&:=\, \sum_{j=1}^{3}
	y_j dx^j, \label{eq:form-phi}\\
	\varphi^\ast(x,y)\,&:=\, \sum_{j=1}^{3}
	(-1)^{j+1}y_j dx^1\wedge \cdots \wedge dx^{j-1} \wedge dx^{j+1} \wedge \cdots \wedge dx^3.  \label{eq:form-phi-star}
\end{align}
Then we say that a current $T\in\mathcal R_2(\Omega \times \R^{3}_y)$
is a {\it generalized Gauss graph} on $\Omega$ if $T=\tau(G,\beta,\eta)$ satisfies the following conditions:
\begin{align}\label{GG1}
	\textrm{$T$ and $\partial T$ are supported on $\Omega  \times S^2$,}
	\\
	\label{GG2} 
	\langle T,\varphi\wedge\omega\rangle=0, \quad \text{ for all } \omega \in \mathcal D^1(\Omega  \times \R^{3}_y), \\
	\label{GG3}
	\langle T,g\varphi^\ast\rangle \geq 0, \quad \text{ for all } g \in C^0_c(\Omega  \times \R^{3}_y)\text{ such that }g\geq 0. 
\end{align} 
We denote by ${\rm curv}_2(\Omega)$ the set of generalized Gauss graphs on $\Omega$.

Condition \eqref{GG2} is equivalent to the orthogonality of $y$ and the enveloping subspace of $\eta(x,y)$ for $\Ha^2$ almost every $(x,y)\in G$ (see \cite[Prop.\,3.1]{Dell97}). The condition \eqref{GG3} fixes the orientation of $G$.

We associate to $T=\tau(G,\beta,\eta)\in {\rm curv}_2(\Omega)$ the stratifications $\eta_0,\eta_1,\eta_2$ as in \eqref{eq:def-strati} and define the Radon measures $|T|,|T_0|$ on $\Omega\times\R^3_y$ by
\begin{align}
	|T| \,:=\, \beta\Ha^2\mres G,\qquad |T_0| \,:=\, \beta |\eta_0|\Ha^2\mres G,\qquad |T_1| \,:=\, \beta |\eta_1|\Ha^2\mres G \label{eq:def-T-T0}
\end{align}
and the subset
\begin{align}
	G^*\,:=\,\{(x,y)\in G\,:\, |\eta_0|(x,y)>0\}, \label{eq:def-Gstar}
\end{align}
where the enveloping subspace of $\eta$ is not vertical.

\subsection{Curvature varifolds}\label{sec:CV}
Let $G(2,3)$ denote the Grassmann manifold of all two-dimensional unoriented planes in $\R^3$. An {\it integral $2$-dimensional varifold} $V$ in $\R^3$ is a Radon measure on $\R^3\times G(2,3)$ of the special form $V=\vf{S,\beta}$, i.e.~it is characterized by  
\begin{align*}
	V(\psi) \,=\, \int_S \psi(x,T_xS) \beta(x)\,d\Ha^2(x), \quad\text{ for all }\psi\in C^0_c(\R^3\times G(2,3)),
\end{align*}
where $S\subset \R^3$ is a $2$-rectifiable set and where $\beta:S\to\N$ is locally $\Ha^2$-integrable. Then $\mu_V:=\beta\Ha^2\mres S$ is a Radon measure on $\R^3$.

Following Hutchinson \cite{Hutc86} an integral $2$-varifold $V=\vf{S,\beta}$ in $\R^3$ is a \emph{curvature varifold} if there exist $V$-measurable functions $A_{ijk}: \R^3\times G(2,3)\to\R$, $1\leq i,j,k\leq 3$ such that for any $\varphi\in C^1(\R^3\times\R^{3\times 3})$ compactly supported with respect to the first variable
\begin{align}
	0\,=\, \int \Big(\sum_jP_{ij} \partial_j\varphi + \sum_{j,k}A_{ijk} \partial^*_{jk}\varphi + \sum_jA_{jij}\varphi\Big)\,dV(x,P),
	\label{eq:Hutch}
\end{align}
where $P(x)=P_{ij}(x)$ denote the orthogonal projection on $T_x S$ and where we have used the notation from Section \ref{sec:surfaces}. Note that the latter equation corresponds to \eqref{eq:Hutch-pInt} for classical surfaces and that the function $A_{ijk}$ generalizes the derivative $\delta_i P_{jk}$ of the projection. In analogy with the representation for the smooth case given in Lemma \ref{lem:HutchK} we define a generalized mean curvature and Gauss curvature for Hutchinson's varifolds.

\begin{definition}\label{def:HutchK}
For an curvature varifold $V$ as above and $(x,y^\perp)\in \spt V$ we define
\begin{align*}
	H_j(x,y^\perp):=\sum_i A_{iji}(x,y^\perp), \quad K(x,y^\perp) \,:=\, \sum_k \tr \cof (A_{ijk}(x,y^\perp))_{ij}.
\end{align*}
\end{definition}
For a curvature varifold $V$ there exist the weak mean curvature $\vec{H}_V$ of $V$ in the sense of Allard \cite{Alla72} and we have for almost all $(x,y^\perp)\in \spt(V)$ that $H_j(x,y^\perp)=\vec{H}_V(x)\cdot \vec{e_j}$. The functions $A_{ijk}$ are $V$-almost everywhere uniquely defined.

Consider now an {\it oriented integral $2$-dimensional varifold} $V=\ovf{S,\tau,\beta_+,\beta_-}$, where $S$ is a $2$-rectifiable set, $\beta_\pm:S\to\N_0$ are $\Ha^2$-measurable with $\beta_++\beta_-\geq 1$ and $\tau(x)$ is an orientation of $T_xS$, and where $V^o=\ovf{S,\tau,\beta_+,\beta_-}$ is characterized by 
$$
V^o(\psi):=\int_S [\psi(x,\tau(x)) \beta_+(x)+\psi(x,-\tau(x))\beta_-(x)]\,d\Ha^2(x),\quad\text{ for all }\psi\in C^0_c(\R^3\times \Lambda^2(\R^3)).
$$
According to \cite[Def.\,3.2]{Dell96} we call $V$ an {\it oriented curvature varifold} if there exist $\Ha^2$-measurable functions $B_{ijk}\colon S \times \Lambda^2(\R^3)\to\R$ such that for all $\psi\in C^1_c(\R^3_x\times\Lambda^2(\R^3))$ and all $1\leq i\leq 3$
\begin{align}
	\int \Big(\sum_j \pi_{ij}\partial_j\psi + \sum_{l<m} B_{ilm}\partial^*_{lm}\psi + \psi \sum_{\substack{k<l\\j}} \frac{\partial \pi_{ij}}{\partial p_{kl}} B_{jkl}\Big)\,dV\,=\,0, \label{eq:def-CV}
\end{align}
where $\partial^*$ denotes derivatives with respect to the second component of $\psi$ and the map $\pi_{ij}\colon \R^3 \times \Lambda^2(\R^3)\to\R$ is given by $\pi_{ij}(x,w) \,:=\, \langle w\mres e_i, w\mres e_j \rangle$.
We notice that $\pi_{ij}(x,\tau)$ is the orthogonal projection on the enveloping subspace of $\tau$ whenever $\tau$ is simple. 

\subsection{Relation between curvature varifolds and generalized Gauss graphs}\label{sec:CV-GG}
Let us associate to a generalized Gauss graph $T=\tau(G,\beta,\eta)\in {\rm curv}_2(\Omega)$ as above the set $S:=\pi_1G\subset \R^3_x$, where $\pi_1:\R^3_x\times\R^3_y\to\R^3_x$ denotes the projection on the first component. By the structure Theorem \cite[Thm.\,2.9]{AnST90} the set $S$ is 2-rectifiable and for any $\Ha^2$-measurable function $\nu:S\to\R^3$ with $\nu(x)\perp T_xS$ for $\Ha^2$-almost all $x\in S$ 
\begin{align*}
	\pi_1|_S^{-1}(x) \,\subset\, \{(x,\nu(x)),(x,-\nu(x))\}
\end{align*}
holds. We then let $V^o_T=\ovf{S,*\nu,\beta(\cdot,\nu),\beta(\cdot,-\nu)}$ be the associated unoriented varifold to $T$. Moreover, we define an $\Ha^2$-measurable function $\gamma:S\to\N$ by
\begin{align*}
	\gamma(x) \,:=\, \beta(x,\nu(x)) + \beta(x,-\nu(x))
\end{align*} 
and the associated integral $2$-varifold $V_T=\vf{S,\gamma}$. \\
We remark that by \cite[Thm.\,4.3]{Dell96} the set $\pi_1(G\setminus G^*)$ has $\Ha^2$-measure zero, where $G^*$ was defined in \eqref{eq:def-Gstar}.

The following proposition relates the two concepts of Hutchinson's curvature varifolds and generalized Gauss graphs.
\begin{proposition}\label{prop:GGG-CV}
Let $T=\tau(G,\beta,\eta)\in {\rm curv}_2(\Omega)$ be given and let $V_T=\vf{S,\gamma}$ be the associated varifold as defined above. If $T$ satisfies $\partial T=0$ and $|T_1|\ll |T_0|$ then $V_T$ is a curvature varifold and the functions $A_{ijk}$ in \eqref{eq:Hutch} and the mean curvature $H$ are given by
\begin{align}
	A_{ijk}(x,y^\perp) \,&=\, \sum_r \xi_0^{ir}(x,y)\big(\bar{\xi}_1^{rj}(x,y^\perp)y_k+\bar{\xi}_1^{rk}(x,y^\perp)y_j\big), \label{eq:Aijk-CV}
	\\ 
	\bar{\xi}_1^{rj}(x,y^\perp)\,&=\, \frac{1}{\gamma(x)}\Big(\beta(x,y)\xi_1^{rj}(x,y)+\beta(x,-y)\xi_1^{rj}(x,-y)\Big), \label{eq:barxi-CV}\\
	H_j(x,y^\perp) \,&=\, \sum_iA_{iji}(x,y^\perp) \,=\, \sum_{i,r} \xi_0^{ir}(x,y)\bar{\xi}_1^{ri}(x,y^\perp)y_j, \label{eq:Hj-CV}
\end{align}
where $x\in S$, $\xi:= \frac{\eta}{|\eta_0|}$ on $G^*$ and where $y\perp T_x S$, $|y|=1$. (Note that the right-hand sides in \eqref{eq:Aijk-CV} and \eqref{eq:Hj-CV} are in fact invariant under $y\mapsto -y$ since $\xi_0(x,-y)=-\xi_0(x,y)$.)
\end{proposition}
\begin{proof}
By \cite[Thm.\,3.1]{Dell96} since $\partial T=0$ and $|T_1|\ll |T_0|$ we have for the functions $B_{ijk}$ in \eqref{eq:def-CV}
\begin{align*}
	B_{ijk}(x,\tau) \,=\, \sum_{l} \eps_{jkl} \langle \xi_1(x,*\tau), \eps_l\wedge (\tau\mres e_i) \rangle,
\end{align*}
where $x\in S$, $\tau=\tau_1\wedge\tau_2$ and $\xi:= \frac{\eta}{|\eta_0|}$ on $G^*$. 

Observe now that $\tau=*y$ and $*\tau=y$ for some  $y\perp T_x S$ with $|y|=1$. We therefore find that $\tau\mres e_i = \sum_r\xi_0^{ir}(x,y)e_r$ and $\eps_l\wedge (\tau\mres e_i) = -\sum_r\xi_0^{ir}(x,y)e_r\wedge\eps_l$, thus
\begin{align*}
	B_{ijk}(x,\tau) \,=\, -\sum_{r,l} \eps_{jkl}\xi_0^{ir}(x,y)\xi_1^{rl}(x,y).
\end{align*}
Comparing \eqref{eq:Hutch} with \eqref{eq:def-CV} for $\psi(\cdot,\tau)=\varphi(\cdot,\pi)$ and using $\frac{\partial \pi_{jk}}{\partial p_{lm}}\,=\, -\eps_{jlm}y_k - \eps_{klm}y_j$ we arrive at  \begin{align*}
	A_{ijk}(x,y^\perp) \,&=\frac{1}{\gamma(x)}\, \sum_{l<m} \frac{\partial \pi_{jk}}{\partial p_{lm}}\left(\beta(x,y)B_{ilm}(x,\tau) -\beta(x,-y)B_{ilm}(x,-\tau)\right)\\
	&=\, \frac{1}{\gamma(x)}\sum_{l<m} (\eps_{jlm}y_k + \eps_{klm}y_j)\,\cdot\,\\
	&\qquad \cdot\sum_{r,s} \left( \beta(x,y)\eps_{lms}\xi_0^{ir}(x,y)\xi_1^{rs}(x,y) -\beta(x,-y)\eps_{lms}\xi_0^{ir}(x,-y)\xi_1^{rs}(x,-y)\right)\\
	&=\, \frac{1}{\gamma(x)}\sum_{r}
	 \left( \beta(x,y)\xi_0^{ir}(x,y)(y_k\xi_1^{rj}(x,y) +y_j\xi_1^{rk}(x,y)) -\right.\\
	 &\qquad\qquad\left. -\beta(x,-y)\xi_0^{ir}(x,-y)(y_k\xi_1^{rj}(x,-y) +y_j\xi_1^{rk}(x,-y))\right)\\
	&=\, \frac{1}{\gamma(x)}\sum_{r} 
	 \xi_0^{ir}(x,y)y_k\Big( \beta(x,y)\xi_1^{rj}(x,y)+\beta(x,-y)\xi_1^{rj}(x,-y)\Big)+\\
	&\qquad  \qquad +\xi_0^{ir}(x,y)y_j\Big( \beta(x,y)\xi_1^{rk}(x,y)+\beta(x,-y)\xi_1^{rk}(x,-y)\Big),
\end{align*}
where we have used that $\xi_0(x,-y)=-\xi_0(x,y)$. This shows \eqref{eq:Aijk-CV}. For \eqref{eq:Hj-CV} observe that $\sum_i \xi_0^{ir}(x,y)y_i=0$ since $\xi_0=*y$.
\end{proof}
%
\section{Proof of theorem \ref{thm:main3}}\label{sec:proof}

In this section we prove Theorem \ref{thm:main3}. %
In order to characterize compactness properties of the sequence $(S_j,\theta_j), j\in\N$ we will associate to any $(S,\theta)\in\Msp$ a rectifiable current, given by the graph of $\theta$ over $S$. To be more precise let
\begin{align}
	G \,:=\, \{(p,\theta(p)) : p \in S\}. \label{eq:def-GG}
\end{align}
Notice that $\theta$ is in general not orthogonal to $S$ and that $G$ therefore is not necessarily a Gauss graph. As above we distinguish the space $\R^3_x$, where the surface $S$ is embedded, and the ambient space $\R^3_y$ of the image of $\theta,\nu$. Consider the parametrization $\pGG\colon S \to \Omega \times \R^3_y$,
\begin{align}\label{eq:def-graph}
	\pGG(p) \,:= \, (p,\theta(p)), \quad\text{ for }p\in S,
\end{align}
of $G$ over the surface $S$. 
 
From the calculations in Lemma \ref{lem:evDtheta} we know that $L(p)$ has eigenvalues $\evla(p),\evmu(p),0$ with an associated positively oriented orthonormal basis $\{v_1(p),v_2(p),\theta(p)\}$ of $\R^3$ given by eigenvectors of $L(p)$. Moreover, by \eqref{eq:evDtheta} the eigenvalues of $L(p)$ are controlled by
\begin{align}
	\int_{S} \big(\evla(p)^2 + \evmu(p)^2\big)\,d\mathcal H^2(p) \,\leq\, 12\int_S Q(L(p)) \,d\hs^{2}(p). \label{eq:eigenval-control}
\end{align}
One key bound to obtain the compactness of the graphs $G_j$ associated to $(S_j,Q_j)$ is the control of their area.
\begin{proposition}\label{prop:jacobian}
For $(S,\theta)\in\Msp$, $L$ as defined above, and the associated graph $G$ as in \eqref{eq:def-GG} we have
\begin{align}
	\Ha^2(G) \,\leq\, \Ha^2(S) + 12 \int_S Q(L(p)) \,d\hs^{2}(p). \label{eq:area-G}
\end{align}
\end{proposition}
\begin{proof}
Choose an orthonormal basis $\tau_1(p),\tau_2(p)$ of $T_pS$ and let us drop for the moment all arguments $p$. We then have
\begin{align}
	D\theta \tau_i \,&=\, \sum_{k=1,2} 
	(v_k\cdot D\theta \tau_i )v_k \,=\, \sum_{k=1,2}
	\lambda_k (v_k\cdot \tau_i)v_k. \label{eq:Dtheta} 
\end{align}
and obtain for the Jacobian of the parametrization $\pGG:S\to G$ of $G$
\begin{align}
	(\jac\pGG)^2 \,&=\, \det\Big( \begin{pmatrix} \tau_i\\ D\theta \tau_i\end{pmatrix}\cdot
	\begin{pmatrix} \tau_j\\ D\theta \tau_j \end{pmatrix}\Big)_{i,j=1,2} \notag\\
	&=\, 1 + |D\theta \tau_1 |^2 + |D\theta \tau_2 |^2 + |D\theta \tau_1 |^2|D\theta \tau_2 |^2 - \big(D\theta \tau_1 \cdot D\theta \tau_2 \big)^2 \notag\\
	&=\, 1 + \evla^2(1-(v_1\cdot\nu)^2) + \evmu^2(1-(v_2\cdot\nu)^2) + \evla^2\evmu^2 (\theta\cdot\nu)^2. \label{eq:jac1}
\end{align}
In particular we deduce that
\begin{align}
	1\,\leq\, |\jac\pGG | \,\leq\,  1+ (\evla^2+\evmu^2). \label{eq:jac2}
\end{align}
By \eqref{eq:eigenval-control} and the area formula we then obtain
\begin{align*}
	\Ha^2(G) \,=\, \int_S |\jac\pGG | \,\leq\, \Ha^2(S) + 12\int_S Q(L(p))\,d\Ha^2(p).
\end{align*}
\end{proof}

We next turn to some estimates related to the current associated to $G$. Let $\tau:=\ast\nu$ and note that $\nu=\tau_1\wedge \tau_2$, where $\{\tau_1(p),\tau_2(p),\nu(p)\}$ is any positively oriented orthornomal basis of $\R^3$. We then consider the tangent $2$-vector field $\xi$ on $G$ and the unit tangent $2$-vector field $\eta$ on $G$, given by
\begin{align}\label{eq:def-xi}
	\xi(p) \,:=\, D\pGG_p(\tau_1(p))\wedge D\pGG_p(\tau_2(p)), \quad \eta(p):=\frac{\xi(p)}{|\xi(p)|},
\end{align}
and define the current $T_G\in\mathcal R_2(\Omega\times \R^3_y)$ by
\begin{align}
	T_G \,:= \,\tau(G,1,\eta), \label{eq:def-TG} 
\end{align}
see Section \ref{sec:Rcurr}.
We first collect some useful information on $\xi$ and its stratifications, cf. \eqref{eq:def-strati}.
\begin{lemma}\label{lem:xi}
Let $\{e_1,e_2,e_3\}$ and $\{\eps_1,\eps_2,\eps_3\}$ denote the standard basis of $\R^3_x$, $\R^3_y$, respectively. Let us further represent $\xi= \xi_0 + \xi_1 + \xi_2$ by its stratifications,
\begin{align*}
	\xi_0 \,&=\, \sum_{1\leq i<j\leq 3} 
	\xi_0^{ij}e_i\wedge e_j, \quad
	\xi_1 \,=\, \sum_{1\leq i,j\leq 3} 
	\xi_1^{ij}e_i\wedge \eps_j, \quad
	\xi_2 \,=\, \sum_{1\leq i<j\leq 3} 
	\xi_2^{ij}\eps_i\wedge \eps_j.
\end{align*}
For convenience, we also set $\xi_0^{ij}=-\xi_0^{ji}$, $\xi_0^{ii}=0$ for $1\leq i,j\leq 3$, $j<i$. We then have
\begin{align}
	\xi_0\,&=\, \tau_1\wedge \tau_2,\quad \xi_0^{ij}\,=\, \tau_{1,i}\tau_{2,j}-\tau_{1,j}\tau_{2,i} \,=\, \sum_{k=1}^3
	\perm_{ijk} \nu_k, \label{eq:rep-xi0}\\
	\xi_1\,&=\, \tau_{1}\wedge D\theta \tau_2 -\tau_{2}\wedge D\theta \tau_1 ,\\
	\xi_1^{ij} \,&=\, \big(\tau_1\otimes D\theta \tau_2 -\tau_2\otimes D\theta \tau_1 \big)_{ij} \label{eq:rep-xi1}\\
	\xi_2\,&=\, D\theta \tau_1 \wedge D\theta \tau_2 ,\quad  \xi_2^{ij}\,=\,  \sum_{k=1}^3
	\big(D\theta \tau_1 \times D\theta \tau_2 \big)_k\perm_{ijk}
	\,=\,\evla\evmu (\theta\cdot \nu) (v_1\wedge v_2)_{ij}.\label{eq:rep-xi2}
\end{align}
Moreover we find
\begin{align}
	|\xi_0|^2 \,&=\, 1,\quad |\xi_1|^2 \,=\, |D\theta \tau_1|^2 + |D\theta \tau_2|^2 \,=\, \sum_{k=1}^2 
	\lambda_k^2\big(1-(v_k\cdot\nu)^2\big),\label{eq:xi-jac1}\\
	|\xi_2|^2 \,&=\, |D\theta \tau_1 \times D\theta \tau_2|^2 \,=\, \evla^2\evmu^2 (\theta\cdot \nu)^2,\label{eq:xi-jac2} \\
	|\xi|\circ\pGG \,&=\, \jac \pGG. \label{eq:xi-jac}
\end{align}
\end{lemma}
\begin{proof}
The assertions follow by straightforward calculations: concerning the identities involving $\lambda_1,\lambda_2$ we recall that
$$
D\theta \tau_i \,=\, \sum_k
	\lambda_k (v_k\cdot \tau_i)v_k
$$
from which the last-hand side of \eqref{eq:rep-xi2}, \eqref{eq:xi-jac1} and \eqref{eq:xi-jac2} follow immediately.
\end{proof}
We next investigate some useful properties of the current $T_G=\tau(G,1,\eta)$ associated to $G$. %
\begin{lemma}
Consider $\varphi^\ast \in\mathcal D^2(\R^{3}_x \times \R^{3}_y)$ and $\varphi \in \mathcal D^1(\R^{3}_x \times \R^{3}_y)$ as in \eqref{eq:form-phi}, \eqref{eq:form-phi-star}. For any $\omega \in \mathcal D^1(\Omega\times \R^3_y)$ there exists a positive constant $C$ such that 
\begin{equation}\label{GG2anow}
	|\langle T_G,\varphi\wedge\omega\rangle| \le C\|\omega\| \Big(\int_S (1- (\theta\cdot\nu)^2)\,d\Ha^2\Big)^{\frac{1}{2}}\Big(\int_S (1+ Q(L))\,d\Ha^2\Big)^{\frac{1}{2}}.
\end{equation}
Moreover, for any $g \in C^0_c(\Omega \times \R^3_y)$ such that $g\geq 0$ we have
\begin{equation}\label{GG2bnow}
	\langle T_G,g\varphi^\ast\rangle \geq 0.
\end{equation}
\end{lemma}

\begin{proof}
For any $\omega \in \mathcal D^1(\Omega\times\R^3_y)$ we have the pointwise estimates
\begin{align}
	|\langle \varphi \wedge \omega, \xi\rangle| \,&=\, \left|\smat{ y\\0}\cdot D\pGG \tau_1 \langle \omega, D\pGG \tau_2 \rangle - \smat{y\\0}\cdot D\pGG \tau_2 \langle \omega, D\pGG\tau_1\rangle \right|\notag\\
	&=\, \left|(y\cdot \tau_1) \langle \omega, D\pGG \tau_2 \rangle - (y\cdot \tau_2) \langle \omega, D\pGG(\tau_1)\rangle\right| \notag\\
	&\leq\, \sqrt{1- (y\cdot \nu)^2} \sqrt{2+ |D\theta  \tau_1 |^2 + |D\theta  \tau_2|^2} |\omega|.
\end{align}
By the area formula we then deduce
\begin{align}
	|\langle T_G,\varphi\wedge\omega\rangle| \,&\le\, \int_G |\langle \varphi\wedge \omega,\xi\rangle | \frac{1}{|\xi|}\,d\Ha^2 \notag\\
	&\leq\, \|\omega\|\int_S \sqrt{1- (\theta\cdot \nu)^2} \sqrt{2+ |D\theta \tau_1|^2 + |D\theta \tau_2|^2} \,d\Ha^2 \notag\\
	&\leq\, C\|\omega\| \Big(\int_S (1- (\theta\cdot\nu)^2)\,d\Ha^2\Big)^{\frac{1}{2}}\Big(\int_S (1+ Q(L))\,d\Ha^2\Big)^{\frac{1}{2}},
\end{align}
by the last equality in \eqref{eq:xi-jac1} and by \eqref{eq:eigenval-control}, which proves the first claim.

We moreover obtain by \eqref{eq:rep-xi0} that
\begin{align}
	\langle \varphi^*, \xi \rangle \,&=\, \langle \varphi^*, \xi_0 \rangle 
	\,=\, \sum_{1\leq i<j \leq 3}\sum_{k=1}^3
	\perm_{ijk} y_k \xi_0^{ij} \,=\, y\cdot \nu. \label{eq:6.21}
\end{align}
Applying once more the area formula, for all $g \in C^0_c(\Omega \times \R^3_y)$ such that $g\geq 0$ we have
$$
\begin{aligned}
\langle T_G,g\varphi^\ast\rangle&=\int_G\langle \varphi^*(p,y),\eta(p,y)\rangle g(p,y)\,d\Ha^2(p,y)\\
&=\int_S(\nu(p) \cdot \theta(p)) g(p,\nu(p))\,d\mathcal  H^2(p)\geq 0,
\end{aligned}
$$
since we have assumed that $\nu \cdot \theta>0$ everywhere on $S$, which completes the proof of \eqref{GG2bnow}.
\end{proof}
We now start with the proof of Theorem \ref{thm:main3} and first show that the graph currents as defined above converge for a subsequence to a generalized Gauss graph.
\begin{proposition}
Consider a sequence $(S_j,\theta_j)\in \Msp$ as in Theorem \ref{thm:main3}, let $G_j$ denote the associated graph of $\theta_j$ over $S_j$, and let $T_j=T_{G_j}$ be the associated currents to $G_j$ as defined above. Then there exist a subsequence $j\to\infty$ (not relabeled) and generalized Gauss graph $T\in {\rm curv}_2(\Omega)$, $T=\tau(G,\beta,\eta)$, such that
\begin{align}
	T_{j} \,\weak\, T. \label{eq:convTj}
\end{align}
For the Radon measures $|T^j_0|, |T_0|$ on $\Omega\times\R^3_y$ defined by
\begin{align*}
	|T_0^j| \,&:=\, |\eta_0^j|\,\Ha^2\mres G_j, \qquad
	|T_0| \,:=\, |\eta_0|\,\beta\Ha^2\mres G
\end{align*}
and the subsequence $j\to\infty $ with \eqref{eq:convTj} we have
\begin{align}
	|T_0^j| \,&\weakstar\, |T^0|\quad\text{ as Radon measures.} \label{eq:convT0j}
\end{align}
\end{proposition}
\begin{proof}
From \eqref{eq:bd-area}, \eqref{eq:Lambda}, and Proposition \ref{prop:jacobian} we deduce that 
\begin{align}
	\mathcal  H^2(G_j) \,&\le\, C\Big(\mathcal H^2(S_j)+\int_{S_j}Q(L_j)\,d\mathcal H^2\Big) \,\leq\, C(1+\Lambda). \label{eq:area-Gj}
\end{align}
Next we notice that $\partial T_j=0$ because $S_j$ has no boundary and $\theta_j \colon S_j \to S^2$ is Lipschitz continuous \cite{GiMS89}. We therefore deduce that the sequence $(T_j)_{j\in\N}$ has uniformly bounded mass and boundary mass. Applying the Federer-Fleming compactness Theorem \ref{comp-federer} we deduce that $T_{j_k} \weak T$ for some $T\in\mathcal R_2(\Omega \times \R^3_y)$.

It remains to show that $T\in {\rm curv}_2(\Omega)$, i.e.\,that \eqref{GG1}, \eqref{GG2} and \eqref{GG3} hold. First of all, since $T_j$ is supported in $\tilde\Omega \times S^2$ for any $j\in \N$ we obtain that $T$ is supported in $\Omega\times S^2$. Moreover, since $\partial T_j=\emptyset$ the convergence as currents also implies $\partial T=\emptyset$.
Therefore, \eqref{GG1} is satisfied by $T$.

Since \eqref{GG2anow} holds for all $T_j$ we deduce that for any $\omega\in \mathcal D^1(\Omega\times \R^3_y)$
\begin{align}
	&\limsup_{j\to\infty}|\langle T_j,\varphi\wedge\omega\rangle|\\ \le\,&  C\|\omega\| \limsup_{j\to\infty}\Big(\int_{S_j} (1- (\theta_j\cdot\nu_j)^2)\,d\Ha^2\Big)^{\frac{1}{2}}\Big(\int_{S_j} (1+ Q(L_j))\,d\Ha^2\Big)^{\frac{1}{2}} \,=\,0, \label{eq:GG2aaa}
\end{align}
where we have used \eqref{eq:bd-area} and \eqref{eq:Lambda}. This shows \eqref{GG2}. 
Similarly, from \eqref{GG2bnow} for $T$ replaced by $T_j$ we obtain in the limit $j\to\infty$ \eqref{GG3}.
This concludes the proof that $T\in {\rm curv}_2(\Omega)$.

For the proof of \eqref{eq:convT0j} we follow \cite[Prop.\,2.8]{AnST90}. By \eqref{eq:6.21} we have that for all $g\in C^0_c(\Omega\times\R^3_y)$
\begin{align*}
	\int_{G_j} |\eta_0^j(x,y)|(y\cdot\nu_j(x))g(x,y)\,d\Ha^2(x,y)\,&=\,\int_{G_j} \langle\varphi^*,\eta^j \rangle g\,d\Ha^2\\
	&=\, \langle T_j,\varphi^*g\rangle \\
	&\to\, \langle T,\varphi^*g\rangle\,=\, \int_G |\eta_0|\beta g \,d\Ha^2.
\end{align*}
Furthermore
\begin{align*}
	\Big|\int_{G_j} |\eta_0^j(x,y)|(1- y\cdot\nu_j(x))g(x,y)\,d\Ha^2(x,y)\Big|\,&=\,\Big|\int_{S_j} (1-\theta_j(x)\cdot\nu_j(x))g(x,\theta_j(x))\,d\Ha^2(x)\Big|\\
	&\leq\, \|g\|_{C^0(\Omega\times\R^3_y)} \int_{S_j} (1-\theta_j\cdot\nu_j)\,d\Ha^2\,\to\, 0
\end{align*}
by \eqref{eq:Lambda}. Together with the previous convergence statement \eqref{eq:convT0j} follows.
\end{proof}
The next lemma collects further properties of the limit Gauss graph $T$.
\begin{lemma}\label{lem:perpT}
The Gauss graph $T=\tau(G,\beta,\eta)$ from \eqref{eq:convTj} satisfies for $\Ha^2$-almost every $(x,y)\in G$
\begin{align}
	\sum_{1\leq i \leq 3} \eta_1^{ij}(x,y)y_i \,&=\, 0\quad&\text{ for all }1\leq j\leq 3, \label{eq:perpT1}\\
	\sum_{1\leq j \leq 3} \eta_1^{ij}(x,y)y_j \,&=\, 0\quad&\text{ for all }1\leq i\leq 3. \label{eq:perpT2}
\end{align}
\end{lemma}
\begin{proof}
By the orthogonality property \eqref{GG2} we deduce (see the proof of \cite[Proposition 2.4]{AnST90}) that 
\begin{align*}
	\langle \eta(x,y) , (y,0)\wedge (0,w)\rangle \,=\,0 \quad\text{for all }w\in\R^3\text{ and }\Ha^2-\text{a.e. }(x,y)\in G. 
\end{align*}
Therefore we deduce that $\sum_{ij}\eta_1^{ij}(x,y)y_iw_j=0$ for all $w\in\R^3$, which implies \eqref{eq:perpT1}.

By \cite[Theorem 2.10]{AnST90} for $\Ha^2$-almost every $(x,y)\in G^*$ there are an embedded $C^1$ surface $\Sigma\subset\R^3$ and a $C^1$ map $\zeta: \Sigma\to S^2$ such that 
\begin{align}
	\zeta(x)=y,\quad \Lambda^2(\Id \oplus d\zeta_x)(*y)\,=\,\eta(x,y) |\Id \oplus d\zeta_x|. \label{eq:repr-form}
\end{align}
By Lemma \ref{lem:xi} we obtain that for $i=1,2,3$ and $*y=\tau_1\wedge\tau_2$
\begin{align*}
	|\Id \oplus d\zeta_x|\sum_j\eta^{ij}_1y_j \,=\, e_i\cdot (\tau_1\otimes D\zeta(x) \tau_2 -\tau_2\otimes D\zeta(x) \tau_1)y
	\,=\, 0,
\end{align*}
since $D\zeta(x)\tau_k\cdot y=D\zeta(x)\tau_k\cdot \zeta(x)=0$ for $k=1,2$ as $\zeta$ maps into $S^2$. This proves \eqref{eq:perpT2}.
\end{proof}
In the following we derive the lower bound \eqref{eq:main3}. We first express the function $Q(L)$ in terms of $\theta$ and $\xi$.
\begin{lemma}\label{lem:Q-xi}
Let $(S,\theta)\in\Msp$ and consider $L,\xi$ as defined in \eqref{eq:cdt-theta3} and \eqref{eq:def-xi}. We then have
\begin{align}
	\tr L \,&=\, \frac{1}{\theta \cdot \nu}(\theta_1(\xi_1^{23}-\xi_1^{32})-\theta_2(\xi_1^{13}-\xi_1^{31})+\theta_3(\xi_1^{12}-\xi_1^{21})) \label{eq:trL-xi}\\
	&=\, \frac{1}{\theta\cdot\nu}\langle\Psi_{\theta},\xi_1\rangle 
	,\label{eq:trL-xi-2}\\
	\tr\cof L \,&=\, \frac{1}{(\theta \cdot \nu)^2}\theta\cdot \cof \xi_1 \theta,\label{eq:Q-xi}
\end{align}
where $L,\theta,\nu$ are evaluated in $p\in S$ and $\xi$ in $(p,\theta(p))$, where $\Psi_{\theta}\,=\sum_{i,k,l}\, \perm_{ikl}\theta_i dx^k\wedge dy^l$, where $\cof \xi_1$ denotes the cofactor matrix of the matrix representation $(\xi_1^{ij})_{1\leq i,j\leq 3}$ of $\xi_1$, and where $\xi_0 : \xi_1$ denote the matrix product between the respective matrix representations.

Moreover, we have
\begin{align}\label{propL1}
	\sum_{j=1}^3
	\xi_1^{ij}\theta_j\,=\,0\quad\text{ for all }i=1,2,3, \qquad
	\sum_{k=1}^3
	\xi_1^{kk}\,=\,0. 
\end{align}
\end{lemma}
\begin{proof}
We first observe that for any $r\in\R$
\begin{align*}
	&(\theta\cdot \nu)\Big(-r \tr\cof L + r^2 \tr L -r^3\Big) \notag\\
	=\, &\det (\tau_1 | \tau_2 | \theta)\det(L-r\Id) \notag\\
	=\, & -r (L\tau_1 \times L\tau_2)\cdot \theta + r^2 (\tau_1\times L\tau_2 -\tau_2\times L\tau_1)\cdot \theta -r^3\theta\cdot\nu
\end{align*}
and we deduce that firstly, by \eqref{eq:aux1}
\begin{align*}
	(\theta\cdot\nu) \tr L \,&=\,   (\tau_1\times L\tau_2 -\tau_2\times L\tau_1)\cdot \theta \\
	&=\, \sum_{i,j,k=1}^3
	\big( \tau_{1,i} e_j\cdot L\tau_2 - \tau_{2,i} e_j\cdot L\tau_1\big) \theta_k \perm_{ijk}\\
	&=\, \sum_{i,j,k=1}^3 
	\xi_1^{ij} \theta_k\perm_{ijk}\,
	=\, \sum_{i<j}\sum_{k=1}^3 
	(\xi_1^{ij}-\xi_1^{ji})\theta_k\perm_{ijk},
\end{align*}
which yields \eqref{eq:trL-xi}.
Secondly, we have by \eqref{eq:rep-xi0} and since $L\tau_i\perp\theta$
\begin{align*}
	(\theta\cdot\nu) \tr\cof L \,&=\, (L\tau_1 \times L\tau_2)\cdot \theta \,=\, \frac{1}{\theta\cdot\nu}(L\tau_1 \times L\tau_2)\cdot \nu \,=\, \frac{1}{\theta\cdot\nu} (L\tau_1 \wedge L\tau_2)\cdot \xi_0. 
\end{align*}
Moreover, by \cite[Prop.\,3.21]{Serr10} we have 
\begin{align*}
	\theta\cdot \cof (\tau_1 \otimes D\theta \tau_2 -\tau_2 \otimes D\theta \tau_1)\theta \,&=\, \det (\tau_1 \otimes D\theta \tau_2 -\tau_2 \otimes D\theta \tau_1 + \theta\otimes \theta)\,=:\,T
\end{align*}
and representing the matrix with respect to the bases $\{v_1,v_2,\theta\}$ from Lemma \ref{lem:evDtheta}
\begin{align*}
	T &=\, \det \left(\begin{matrix} L\tau_2\cdot v_1 & L\tau_2\cdot v_2 & \tau_1\cdot\theta\\
	-L\tau_1\cdot v_1 & -L\tau_1\cdot v_2 & \tau_2\cdot\theta\\
	0& 0 & \nu\cdot\theta\end{matrix}\right)\,=\, (\nu\cdot\theta)\det \left(\begin{matrix} L\tau_1\cdot v_1 & L\tau_1\cdot v_2 \\
	L\tau_2\cdot v_1 & L\tau_2\cdot v_2 \end{matrix}\right)\\
	&=\, (\nu\cdot\theta)\evla\evmu \det \left(\begin{matrix} \tau_1\cdot v_1 & \tau_1\cdot v_2 \\
	\tau_2\cdot v_1 & \tau_2\cdot v_2 \end{matrix}\right)
	\,=\, (\nu\cdot\theta)^2 \evla\evmu \,=\, (\theta\cdot\nu)^2 \tr\cof L.
\end{align*}

The identities \eqref{propL1} follow from the symmetry of $L$ and since $L\theta=0$.
\end{proof}

We next rephrase the functional $\Qfu$ defined in \eqref{eq:def-Qfu} as a functional on currents in $\Omega\times \R^3$.
\begin{lemma}\label{lem:Qfu-currents}
Fix $y\in S^2$ and consider
\begin{align}
	\Xsp_y \,:=\, \{ \zeta\in \Lambda_1(\R^3_x)\wedge\Lambda_1(\R^3_y)\,:\, \sum_{\alpha=1}^3{\zeta}^{k\alpha}y_\alpha=\sum_{h=1}^3{\zeta}^{hh}=0\text{ for all }k=1,2,3\}. \label{eq:def-Xsp}
\end{align}
We define $f_y:\Xsp_y \to [0,\infty]$ by
\begin{align}
	f_y(\zeta) \label{eq:def-f}
	\,:=\, \frac{1}{4}\langle \Psi_y,\zeta\rangle^2-\frac{1}{6}y \cdot {\rm cof}\,\zeta y.
\end{align}
For $(S,\theta)\in\Msp$ consider the graph $G$ as in \eqref{eq:def-GG} and the simple unit 2-vector field $\eta$ as in \eqref{eq:def-xi}.
Then
\begin{align}
	\int_S Q(L(p))\,d\Ha^2(p)
	\,=\, \int_{G^\ast} \frac{1}{|\frac{\eta_0}{|\eta_0|}\wedge y|^2} f_y\bigg(\frac{\eta_1}{|\eta_0|}(x,y)\bigg)|\eta_0|(x,y) \,d\Ha^2(x,y). \label{eq:Qfu-1}
\end{align}
\end{lemma}

\begin{proof}
By \eqref{propL1} we have $\eta_1(x,y)\in \Xsp_y$ for all $(x,y)\in G$. Moreover \eqref{eq:trL-xi} and \eqref{eq:Q-xi} yield
\begin{align}
	\int_S Q(L(p))\,d\Ha^2(p) \,
	&=\, \int_S \Big(\frac{1}{4}(\tr\,L(p))^2-\frac{1}{6}\tr {\rm cof}\,L(p) \Big)\,d\Ha^2(p) \notag\\
	&=\, \int_S \frac{1}{(\theta \cdot \nu)^2}\Big(\frac{1}{4}(\langle \Psi_y,\xi_1\rangle)^2 -\frac{1}{6}
	\theta\cdot \cof \xi_1 \theta\Big)\,d\Ha^2 \notag\\
	&=\, \int_{G^\ast} \frac{1}{|\frac{\eta_0}{|\eta_0|}\wedge y|^2}\bigg[\frac{1}{4}\Big(\langle \Psi_y,\frac{\eta_1}{|\eta_0|}\rangle\Big)^2-\frac{1}{6}y \cdot \cof\frac{\eta_1}{|\eta_0|}y\bigg] {|{\eta}_0|}\,d\Ha^2(x,y)
\end{align}
where $\eta$ is evalutated at $(x,y)$, and where we have used that $G=\pGG(S)$ and
$$
	\eta \,=\, \frac{\xi}{|\xi|}, \quad |\xi|=\frac{1}{|{\eta}_0|}=|\det D\pGG|
$$
which yields the conclusion.
\end{proof}
We will next show that $f$ has suitable convexity properties. For this it is more convenient to represent $\zeta\in \Lambda_1(\R^3_x)\wedge\Lambda_1(\R^3_y)$ as a vector in $\R^9$ and $\zeta\mapsto f_y(\zeta)$  as a quadratic form.
\begin{lemma}\label{lem:qform}
Let us fix $(x,y)\in \Omega\times S^2$ and define an isomorphism $\Lambda_1(\R^3_x)\wedge\Lambda_1(\R^3_y)\to\R^9$, $\zeta\mapsto u=u{[\zeta]}\in\R^9$ by
\begin{align*}
	u \,:=\, ( \zeta^{11}, \zeta^{12}, \zeta^{13}, \zeta^{21}, \zeta^{22}, \zeta^{23}, \zeta^{31}, \zeta^{32}, \zeta^{33})^T. 
\end{align*}
Then $f_y$ as in \eqref{eq:def-f} transforms to a quadratic form $u\,\mapsto\, u\cdot A_y u$ on $\R^9$, where 
\begin{align*}
	A_y\,:=\, 
	\begin{pmatrix}
		0&0&0&0&-y_3 ^2&y_2 y_3 &0&y_2 y_3 &-y_2 ^2\\
		0&3y_3 ^2&-3y_2 y_3 &-2y_3 ^2&0&2y_1y_3 &2y_2 y_3 &-3y_1y_3 &y_1y_2 \\
		0&-3y_2 y_3 &3y_2 ^2&2y_2 y_3 &y_1y_3 &-3y_1y_2 &-2y_2 ^2&2y_1y_2 &0\\
		0&-2y_3 ^2&2y_2 y_3 &3y_3 ^2&0&-3y_1y_3 &-3y_2 y_3 &2y_1y_3 &y_1y_2 \\
		-y_3 ^2&0&y_1y_3 &0&0&0&y_1y_3 &0&-y_1^2\\
		y_2 y_3 &2y_1y_3 &-3y_1y_2 &-3y_1y_3 &0&3y_1^2&2y_1y_2 &-2y_1^2&0\\
		0&2 y_2 y_3 &-2y_2 ^2&-3y_2 y_3 &y_1y_3 &2y_1y_2 &3y_2 ^2&-3y_1y_2 &0\\
		y_2 y_3 &-3y_1y_3 &2y_1y_2 &2y_1y_3 &0&-2y_1^2&-3y_1y_2 &3y_1^2&0\\
		-y_2 ^2&y_1y_2 &0&y_1y_2 &-y_1^2&0&0&0&0
	\end{pmatrix}.
\end{align*}
Next let
\begin{align*}
	v^{(-1)} &:= \begin{pmatrix} y_1 ^2 -1 \\y_1 y_2  \\y_1 y_3\\y_1 y_2  \\y_2 ^2 -1 \\y_2 y_3 \\y_1 y_3 \\y_2 y_3 \\y_3 ^2-1  \end{pmatrix},\, 
	v^{(5)} :=\begin{pmatrix}0\\-y_3 \\y_2 \\y_3 \\0\\-y_1 \\-y_2 \\y_1 \\0\end{pmatrix},\, 
	v^{(1)}_1 :=\begin{pmatrix}2 y_1y_2y_3 \\ y_2 ^2 y_3- y_3\\y_2 y_3^2 -y_1^2y_2 \\ y_2 ^2y_3 - y_3 \\0\\
	y_1  -y_1 y_2 ^2 \\y_2 y_3^2 -y_1^2y_2  \\y_1  -y_1 y_2 ^2 \\-2y_1 y_2 y_3 \end{pmatrix},\, 
	v^{(1)}_2 :=\begin{pmatrix}y_1 (y_2^2-y_3^2) \\y_2 ^3-y_2 \\y_2 ^2y_3 +y_1 ^2y_3 \\y_2 ^3-y_2 \\y_1 -y_1 y_2 ^2\\0\\
	 y_2 ^2y_3 +y_1 ^2y_3\\0\\-y_1 ^3-y_1 y_2 ^2\end{pmatrix},\\
	v^{(0)}_1 &:= (y,0,0)^T,\quad
	v^{(0)}_2 := (0,y,0)^T,\quad
	v^{(0)}_3 := (0,0,y)^T,\\
	v^{(0)}_4 &:= (y_1 e_1,y_2 e_1, y_3 e_1)^T,\quad
	v^{(0)}_5 := (y_1 e_2, y_2 e_2, y_3 e_2)^T,\quad
	v^{(0)}_6 := (y_1 e_3, y_2 e_3, y_3 e_3)^T,
\end{align*}
where $\{e_1,e_2,e_3\}$ denotes the standard basis of $\R^3$.
Then these vectors are eigenvectors of $A_y$ with corresponding eigenvalues $-1,0,1,5$, more precisely
$$
	A_yv^{(-1)}=-v^{(-1)}, \quad A_yv^{(5)}=5v^{(5)}, \quad A_yv_i^{(1)}=v_i^{(1)}, \quad A_yv_j^{(0)}=0
$$
for $i=1,2$ and $j=1,\dots,6$. Moreover, 
\begin{align*}
	\R^9 \,&=\, \spann \{v^{(-1)},v^{(5)},v_1^{(1)},v^{(1)}_2,v^{(0)}_1,\dots,v^{(0)}_6\},\\
	5 \,&=\, \dim \spann \{v^{(0)}_1,\dots,v^{(0)}_6\}.
\end{align*}
By the bijection $\zeta^1\,\mapsto\, u[\zeta^1]$ the space $\Xsp_y$ from \eqref{eq:def-Xsp} transforms to
\begin{align}
	\tilde{\Xsp_y}\,&:=\, \big\{u\in \R^{9}\,:\, u\perp \spann\{ (y,0,0)^T, (0,y,0)^T, (0,0,y)^T,(e_1,e_2,e_3)^T\}\big\}. \label{eq:trsfd-Xy}
\end{align}
Then
\begin{align}
	v^{(-1)}\,&\perp\, \tilde{\Xsp_y}, \label{eq:Xsp-cap1}\\
	\spann \{v^{(0)}_j\,:\, j=1,2,3\} \,&\perp\, \tilde{\Xsp_y}.\label{eq:Xsp-cap0}
\end{align}
\end{lemma}
\begin{proof}
The claims follow by straightforward calculations. For \eqref{eq:Xsp-cap1} observe that
\begin{align}
	v^{(-1)} \,=\,  y_1(y,0,0)^T + y_2(0,y,0)^T + y_3(0,0,y)^T -(e_1,e_2,e_3)^T. \label{eq:v-1}
\end{align}
\end{proof}
The previous lemma shows that $u\mapsto u\cdot A_yu$ behaves nicely outside the eigenspaces corresponding to the eigenvalues $-1$, $0$. The relevant space $\Xsp_y$ is orthogonal to $v^{(-1)}$ and $v^{(0)}_{1,2,3}$. For a generalized Gauss graph, we obtain that $\Xsp_y$ is orthogonal to the full zero eigenspace, too. In our situation this is not the case, but the projection onto that eigenspace is small, see Lemma \ref{lem:4.9}. We therefore consider a suitable modification of the quadratic form $u\mapsto u\cdot A_yu$.
\begin{proposition}\label{prop:f}
Define for $(x,y)\in \Omega\times S^2$ mappings $F_y:\R^9\,\to\,\R$ by
\begin{align}
	F_y(u)\,:=\, u\cdot A_y u + |\pi_0 u|^\frac{3}{2} + 2|\pi_{-1}u|^2, \label{eq:def-Fu}
\end{align}
and
\begin{align}
	\tilde{f}\,:\, \Omega\times S^2\times \big(\Lambda_1(\R^3_x)\wedge\Lambda_1(\R^3_y)\big)\,\to\,\R,\qquad
	\tilde{f}(x,y,\zeta)\,:=\, F_y(u[\zeta]), \label{eq:def-tildef}
\end{align}
where $A_y$, and $u[\cdot]$ have been defined in Lemma \ref{lem:qform} and where $\pi_0, \pi_{-1}$ denote the orthogonal projections on $\spann\{v_j^{(0)},\,j=1,\dots,6\}$ and on $\spann\{v^{(-1)}\}$, respectively. 

Then $\tilde{f}$ is continuous, nonnegative, convex in the third variable and has uniform super-linear growth in the third variable in the sense that
\begin{align}
	\tilde{f}(x,y,\zeta) \,\geq\, |u[\zeta]-\pi_{0}u[\zeta]|^2  + |\pi_0 u[\zeta]|^\frac{3}{2}. \label{eq:growth}
\end{align}
\end{proposition}
\begin{proof}
The continuity of $\tilde{f}$ is clear from the definition, to prove the other claims it is sufficient to show that $F_y$ is nonnegative, convex, and satisfies \eqref{eq:growth}.\\
We let $\pi_1,\pi_5,\pi_0,\pi_{-1}$ denote the orthogonal projection on the corresponding eigenspaces and compute for an arbitrary $u\in\R^9$
\begin{align*}
	F_y(u)\,&=\, - |\pi_{-1}u|^2  +|\pi_{1}u|^2  + 5|\pi_{5}u|^2  + |\pi_0 u|^\frac{3}{2} + 2|\pi_{-1}u|^2
	\\&=\, |\pi_{-1}u|^2  +|\pi_{1}u|^2  + 5|\pi_{5}u|^2  + |\pi_0 u|^\frac{3}{2}.
\end{align*}
We therefore deduce that $F_y$ is convex, and that
\begin{align*}
	F_y(u)\,&\geq\, |u-\pi_{0}u|^2  + |\pi_0 u|^\frac{3}{2}
\end{align*}
which proves \eqref{eq:growth}.
\end{proof}
The following lemma characterizes the norm of $\pi_0 u$ in \eqref{eq:def-Fu}. Note that $\{{v}^{(0)}_1,\dots,{v}^{(0)}_6\}$ is not an orthonormal system, in particular it is in general not true that $v^{(0)}_{i+3}\cdot v^{(0)}_j=0$ for $1\leq i,j\leq 3$.
\begin{lemma}\label{lem:4.9}
Let $y\in S^2$. For $u\in \tilde{\Xsp}_y$ and the orthogonal projection 
$$
\pi_0:\R^9\,\to\, \spann\{{v}^{(0)}_1,\dots,{v}^{(0)}_6\}
$$ 
on the zero eigenspace of $A_y$ as above we have
\begin{align}
	|\pi_0 u|^2\,=\, \sum_{i=4}^6 ({v}^{(0)}_i\cdot u)^2. \label{eq:norm-pi0}
\end{align}
\end{lemma}
\begin{proof}
Let $V\in \R^{9\times 3}$ denote the matrix that consists of the row vectors ${v}^{(0)}_1,{v}^{(0)}_2,{v}^{(0)}_3$ and let $\tilde{V}\in \R^{9\times 3}$ denote the matrix that consists of the row vectors ${v}^{(0)}_4,{v}^{(0)}_5,{v}^{(0)}_6$. We observe that the corresponding rows form two orthonormal systems and thus
\begin{align}
	V^T V \,=\, \tilde{V}^T\tilde{V} \,=\, \Id_3. \label{eq:VVT}
\end{align}
Write $u_0=\pi_0 u$, then for suitable $\alpha,\tilde{\alpha}\in\R^3$ we have
\begin{align*}
	u_0\,=\, \left(V \,\,\tilde{V}\right)\left(
	\begin{matrix}
		\alpha\\ \tilde{\alpha}
	\end{matrix}\right).
\end{align*}
We further obtain
\begin{align*}
	\tilde{V}^TV \,=\, ( v^{(0)}_{i+3}\cdot v^{(0)}_j )_{1\leq i,j\leq 3} \,=\, (y_iy_j)_{1\leq i,j\leq 3} \,=\, yy^T.
\end{align*}
Since $u\in \tilde{\Xsp}_y$ we deduce that $u_0\perp v^{(0)}_i$, $i=1,2,3$, and therefore
\begin{align*}
	0\,&=\, V^T u_0 \,=\,  V^T\left(V \,\,\tilde{V}\right)\left(
	\begin{matrix}
		\alpha\\ \tilde{\alpha}
	\end{matrix}\right)
	\,=\, \left(\Id \,\,V^T\tilde{V}\right)\left(
	\begin{matrix}
		\alpha\\ \tilde{\alpha}
	\end{matrix}\right) \,=\, \alpha + (\tilde{\alpha}\cdot y)y,
\end{align*}
which implies $\alpha=-(\tilde{\alpha}\cdot y)y$ and $\alpha\cdot y =-\tilde{\alpha}\cdot y$. Similarly we deduce
\begin{align*}
	\tilde{V}^T u_0 \,&=\, ({\alpha}\cdot y)y +  \tilde{\alpha}.
\end{align*}
Putting everything together we obtain
\begin{align*}
	|u_0|^2 \,&=\, (V\alpha +\tilde{V}\tilde{\alpha})^T(V\alpha +\tilde{V}\tilde{\alpha})\\
	&=\, \alpha^TV^TV\alpha + \alpha^T (\tilde{V}^TV)^T\tilde{\alpha}+ \tilde{\alpha}^T \tilde{V}^TV{\alpha} + \tilde{\alpha}^T\tilde{V}^T\tilde{V}\tilde{\alpha}\\
	&=\, |\alpha^2| + 2(\alpha\cdot y)(\tilde{\alpha}\cdot y) + |\tilde{\alpha}|^2\\
	&=\, (\alpha\cdot y)^2 + 2(\alpha\cdot y)(\tilde{\alpha}\cdot y) +|\tilde{\alpha}|^2\,=\, |\tilde{V}^T u_0|^2
	\,=\, \sum_{i=4}^6 ({v}^{(0)}_i\cdot u_0)^2.
\end{align*}
As ${v}^{(0)}_i\cdot (u-u_0)\,=\,0$ the assertion follows.
\end{proof}
We are now ready to complete the proof of Theorem \ref{thm:main3}.

\begin{proof}[Proof of \eqref{eq:main3}]
Let us define $\xi^j:= \frac{\eta^j}{|\eta_0^j|}$ First, we claim that 
\begin{align}
	\int_{G_j}\tilde f\big(x,y,\xi_1^j\big)|\eta_0^j|\,d\mathcal H^2 \,&\leq\, \Lambda+1 \quad\text{ for all }j\leq j_0, \label{eq:1:69-bound}\\
	\liminf_{j\to+\infty}\int_{G_j}\tilde f\big(x,y,\xi_1^j\big)|\eta_0^j|\,d\mathcal H^2\,&\leq\, \liminf_{j\to+\infty} \Qfu_{\e_j}(S_j,\theta_j). \label{claim1:69}
\end{align}
In fact, we have
\begin{align*}
	\tilde{f}_y(\xi_1^j) \,=\, f_y(\xi_1^j) + |\pi_0 u[\xi_1^j]|^\frac{3}{2} + 2|\pi_{-1}u[\xi_1^j]|^2
\end{align*}
and \eqref{eq:Lambda}, \eqref{eq:Qfu-1} and $|y|=1$ for $(x,y)\in {\rm spt}\,G_j$ imply
\begin{align}
	\Lambda\,\geq\, \int_{S_j} Q(L_j(p))\,d\mathcal H^2(p) \,\ge\, \int_{G_j} f_y\big(\xi_1^j\big)|\eta_0^j|\,d\mathcal H^2. \label{eq:6.32bis}
\end{align}
Next for $(x,y)\in \spt G_j$ and for $u^j=u[\xi_1^j]$ we deduce by \eqref{eq:Xsp-cap1} that 
\begin{align}
	\pi_{-1}u^j \,=\, 0. \label{eq:6.32-bbis}
\end{align}
Moreover, by the definition of $u^j$ and $v^{(0)}_{4,5,6}$ and by \eqref{eq:rep-xi1}
\begin{align*}
	 \sum_{k=4}^6 \left({v}^{(0)}_k\cdot u^j\right)^2 \circ \pGG^j \,&=\,  \sum_{k=1}^3 \Big(\sum_{i=1}^3 y_i(\xi_1^j)^{ik}\Big)^2  \circ \pGG^j\\
	 &=\, |(\theta_j\cdot \tau_{j,1})D\theta_j \tau_{j,2}-(\theta_j\cdot\tau_{j,2})D\theta_j \tau_{j,1}|^2
\end{align*}
where $\pGG^j$ is the parametrization of $G_j$ as in \eqref{eq:def-graph}.
Together with \eqref{eq:norm-pi0} this yields
\begin{align*}
	\int_{G_j} \big|\pi_0 u^j[\xi_1^j]\big|^{\frac{3}{2}}|\eta_0^j|
	\,&=\, \int_{G_j} \big|\sum_{k=4}^6 \left({v}^{(0)}_k\cdot u^j[\xi_1^j]\right)^2\big|^\frac{3}{2}|\eta_0^j|\,d\mathcal H^2\\
&=\int_{S_j}\sum_{k=1}^3\big|\theta_j\cdot \tau_{j,1}D\theta_{j,k}\tau_{j,2}-\theta_j\cdot \tau_{j,2}D\theta_{j,k}\tau_{j,1}\big|^\frac{3}{2}\,d\mathcal H^2\\
&\le c\int_{S_j}(1-(\theta_j\cdot \nu_j)^2)^\frac{3}{4}(\lambda_{1,j}^2+\lambda_{2,j}^2)^\frac{3}{4}\,d\mathcal H^2\\
&\le c\Big(\int_{S_j}(1-(\theta_j\cdot \nu_j)^2)^3\,d\mathcal H^2\Big)^\frac{1}{4}\Big(\int_{S_j}(\lambda_{1,j}^2+\lambda_{2,j}^2)\,d\mathcal H^2\Big)^\frac{3}{4}\\
&\le \tilde c \sqrt{\e_j}\Big(\int_{S_j}\frac{1}{\e_j^2}(1-\theta_j\cdot \nu_j)\,d\mathcal H^2\Big)^\frac{1}{4}\Big(\int_{S_j}Q(L_j(p))\,d\mathcal H^2\Big)^\frac{3}{4}\stackrel{j\to +\infty}{\to}0
\end{align*}
and this proves together with \eqref{eq:6.32bis} and \eqref{eq:6.32-bbis} the statements \eqref{eq:1:69-bound}, \eqref{claim1:69}.

Consider now for an integral current $T=\tau(G,\beta,\eta)$  in $\Omega \times \R^3_y$ the functional
$$
H(T):=\int_{G^*}\tilde{f}\big(x,y,\frac{\eta_1}{|\eta_0|}\big)|\eta_0|\beta\,d\mathcal H^2
$$
where $G^*=\{(x,y)\in G : |\eta_0|(x,y)>0\}$ and $\tilde{f} \colon \Omega \times \R^3_y\times \Lambda^2(\R^3_x\times \R^3_y) \to [0,+\infty]$ was defined in \eqref{eq:def-tildef}. By Proposition \ref{prop:f} the function $\tilde{f}$ is continuous, convex in the third component and has uniformly superlinear growth in the third component. Let now $T_j=\tau(G_j,1,\eta^j)$ be the graph currents associated to $(S_j,\theta_j)$ and let $|T_0^j|:=|\eta_0^j|\mathcal H^2\mres G_j$. By \eqref{eq:convT0j} we have $|T_0^j|\weakstar |T^0|$ as Radon measures. Next, we consider the measure-function pairs 
$$
\big(|T_0^j|,\xi_1^j\big).
$$
By \eqref{eq:1:69-bound} and Proposition \ref{prop:f} all assumptions of Theorem 4.4.2 of \cite{Hutc86} are satisfied. Therefore, up to a subsequence, 
\begin{align}
\big(|T_0^j|,\xi_1^j\big) \to (|T_0|,g), \quad g\in L^1_{\rm loc}(|T_0|,\Lambda^2(\R^3_x\times\R^3_y)), \label{eq:conv-T0g}
\end{align}
in the sense of measure-function pairs, that is
\begin{equation}\label{convH}
\lim_{j\to+\infty}\int \langle\phi,\xi_1^j\rangle\,d|T_0^j| \,=\, \int \langle\phi, g\rangle \,d|T|, \quad \textrm{for every }\phi\in C^0_c(\Omega\times\R^3_y,\Lambda^2(\R^3_x\times\R^3_y)).
\end{equation}
Moreover, by \eqref{eq:convT0j}, Proposition \ref{prop:f}, \eqref{claim1:69}, and \cite[Thm.\,4.4.2]{Hutc86} we deduce 
\begin{align}
	\int \tilde{f}(x,y,g(x,y))\,d|T_0| \,&\le\, \liminf_{j\to+\infty}\int \tilde{f}\big(x,y,\xi_1^j\big)\,d|T_0^j| \notag\\
	&\le\, \liminf_{j\to+\infty} \Qfu_{\eps_j} (S_j,\theta_j). \label{eq:liminf-hutch}
\end{align}
From \eqref{convH}, $T_j \weak T$, and \eqref{eq:conv-T0g} we obtain
\begin{align*}
	\int \langle\phi,g\rangle |\eta_0| \,d|T| \,&=\, \int \langle\phi, g\rangle \,d|T_0| \\
	 &=\, \lim_{j\to+\infty}\int \langle\phi,\frac{\eta^j_1}{|\eta_0^j|}\rangle\,d|T_0^j| \,=\,
	 \lim_{j\to\infty} \int \langle\phi ,\eta_1^j\rangle \,d|T^j| \,=\, \int \langle\phi,\eta_1 \rangle\,d|T|
\end{align*}
for all $\phi\in C^0_c(\Omega\times\R^3_y,\Lambda^2(\R^3_x\times\R^3_y))$.
Thus we further obtain $|T_1|\ll |T_0|$ and, by \eqref{eq:liminf-hutch}
\begin{align}
	\liminf_{j\to+\infty} \int_{S_j} Q(L_j(p))\,d\mathcal H^2(p) \ge \int_{G^*} \tilde{f}\Big(x,y,\frac{\eta_1}{|\eta_0|}\Big)|\eta_0(x,y)|\beta(x,y)\,d\mathcal H^2(x,y). \label{eq:energy-000}
\end{align}
On $G^*$ we set $\xi:= \frac{\eta}{|\eta_0|}$ and $u=u[\xi_1]$. Then \eqref{eq:def-Fu}, \eqref{eq:def-tildef} yield
\begin{align}
	\tilde{f}\big(x,y,\xi_1\big) \,\geq\, u\cdot A_y u
	\,=\, \frac{1}{4}\langle \Psi_y,\xi_1\rangle^2 - \frac{1}{6} y\cdot \cof\xi_1 y, \label{eq:energy-001}
\end{align}
where $\Psi_y$ and $\cof\xi_1$ are as in Lemma \ref{lem:Q-xi}. 
From the definition of $\Psi_y$ we deduce that
\begin{align}
	\langle \Psi_y,\xi_1\rangle\,=\, \sum_{k,l}(*y)^{kl}\xi_1^{kl}\,=\, \sum_{k,l}\xi_0^{kl}\xi_1^{kl}. \label{eq:energy-002}
\end{align}
We next obtain for the Gaussian curvature $K$ and the mean curvature $H$ of the curvature varifold $V=V_T$ associated to $T$ (see Definition \ref{def:HutchK})  from Proposition \ref{prop:GGG-CV} and \eqref{eq:Aijk-CV}, \eqref{eq:Hj-CV}, and \eqref{eq:lemLA-2} that
\begin{align}
	\frac{1}{4}H^2(x,y^\perp) - \frac{1}{6}K(x,y^\perp) \,&=\, \frac{1}{4}\sum_j \Big(\sum_i A_{iji}(x,y^\perp)\Big)^2 - \frac{1}{6}\sum_k \tr\cof (A_{ijk}(x,y^\perp))_{ij} \notag\\
	&=\, \frac{1}{4}\Big(\sum_{i,r}\xi_0^{ir}(x,y)\bar{\xi}_1^{ri}(x,y^\perp)\Big)^2 - \frac{1}{6} y\cdot \cof \bar{\xi}_1(x,y^\perp) y \notag\\
	&=\, \frac{1}{4}(\xi_0(x,y):\bar{\xi}_1(x,y^\perp))^2 - \frac{1}{6}\tr \cof \bar{\xi}_1(x,y^\perp), \label{eq:energy-003}
\end{align}
where we have used \eqref{eq:perpT1}, \eqref{eq:perpT2} and \eqref{eq:lemLA-3}.

Now for $x,y$ with $T_x S=y^\perp$ fixed, $u=u[\bar{\xi}_1(x,y^\perp)]$ we find as in Lemma \ref{lem:qform} that 
\begin{align*}
	\frac{1}{4}\big(\xi_0(x,y):\bar{\xi}_1(x,y^\perp)\big)^2 - \frac{1}{6}\tr \cof \bar{\xi}_1(x,y^\perp) \,=\, u\cdot A_y u.
\end{align*}
For the eigenfunctions $v^{(0)}_i, v^{(-1)}$ of $A_y$ as identified in Lemma \ref{lem:qform} we obtain by \eqref{eq:perpT1} that  $u\cdot v^{(0)}_i=\sum_{j=1}^3\bar{\xi}_{ji}y_{j-3}=0$ for $i=4,5,6$.\\
By \eqref{eq:perpT2} we similarly obtain $u\cdot v^{(0)}_i=\sum_{j=1}^3\bar{\xi}_{ij}y_{j}=0$ for $i=1,2,3$. Finally we also have $u\perp v^{(-1)}$, since for any $\varphi\in C^0_c(\Omega\times \R^3_y)$
\begin{align*}
	&\int \varphi(x,y)\langle (e_1\wedge \eps_1 +e_2\wedge \eps_2 +e_3\wedge \eps_3),\eta_1\rangle\,d|T|(x,y) \\
	=\,& 
	\lim_{j\to\infty} \int \varphi(x,y)\langle (e_1\wedge \eps_1 +e_2\wedge \eps_2 +e_3\wedge \eps_3),\eta_1^j\rangle \,d|T^j|(x,y) \,=\, 0
\end{align*}
by \eqref{propL1} and \eqref{eq:v-1}. 

Therefore $u \mapsto u\cdot A_y u$ is (strictly) convex and thus also the right-hand side in \eqref{eq:energy-003} is convex. This implies
\begin{align}
	\frac{1}{4}H^2(x,y^\perp) - \frac{1}{6}K(x,y^\perp) \,&\leq\, \frac{\beta(x,y)}{\gamma(x)}\Big(\frac{1}{4}\big(\xi_0(x,y):{\xi}_1(x,y)\big)^2 - \frac{1}{6}\tr \cof {\xi}_1(x,y)\Big) + \notag\\
	&\qquad + \frac{\beta(x,-y)}{\gamma(x)}\Big(\frac{1}{4}\big(\xi_0(x,y):{\xi}_1(x,-y)\big)^2 - \frac{1}{6}\tr \cof {\xi}_1(x,-y)\Big). \label{eq:energy-004}
\end{align}
We therefore deduce for any $\Ha^2$-measurable unit field $y=y(x)$ with $y(x)\perp T_x S$ for $\Ha^2$-almost all $x\in S$ that
\begin{align*}
	&\int \Big(\frac{1}{4}H^2 - \frac{1}{6}K\Big)dV \\
	=\,& \int_S \Big(\frac{1}{4}H^2(x,T_xS) - \frac{1}{6}K(x,T_xS)\Big)\gamma(x) d\Ha^2(x)\\
	\leq\, & \int_S \frac{\beta(x,y(x))}{\gamma(x)}\Big(\frac{1}{4}\big(\xi_0(x,y(x)):{\xi}_1(x,y(x))\big)^2 - \frac{1}{6}\tr \cof {\xi}_1(x,y(x))\Big)\gamma(x)\,d\Ha^2(x) + \notag\\
	&\qquad + \int_S \frac{\beta(x,-y(x))}{\gamma(x)}\Big(\frac{1}{4}\big(\xi_0(x,y(x)):{\xi}_1(x,-y(x))\big)^2 - \frac{1}{6}\tr \cof {\xi}_1(x,-y(x))\Big)\gamma(x)\,d\Ha^2(x) \\
	{\leq}\,& \int_{G^*} \Big(\frac{1}{4}\big(\xi_0(x,y):{\xi}_1(x,y)\big)^2 - \frac{1}{6}\tr \cof {\xi}_1(x,y)\Big) |\eta_0(x,y)|\beta(x,y)\,d\Ha^2(x,y)\\
	\leq\, &  \int_{G^*} \tilde{f}\big(x,y,\xi_1(x,y)) |\eta_0(x,y)|\beta(x,y)\,d\Ha^2(x,y)\\
	\leq\, & \liminf_{j\to+\infty} \int_{S_j} Q(L_j(p))\,d\mathcal H^2(p),
\end{align*}
where we have used \eqref{eq:energy-001} and \eqref{eq:energy-000} in the last two estimates.

It remains to show that we can identify the limit $\mu$ in \eqref{eq:conv-Chi} with the mass measure $\mu_{V_T}$ of the varifold $V_T$ associated to $T$. We first obtain for any $\psi\in C^0_c(\Omega)$ by the co-area formula and by \eqref{eq:convT0j}
\begin{align}
	\mu(\psi) \,&=\, \lim_{j\to\infty} \int_{S_j} \psi(x)\,d\Ha^2(x) \,=\, \lim_{j\to\infty} \int_{G_j} \psi(x)|\eta_0^j|(x,y)\,d\Ha^2(x,y) \notag \\
	&=\, \int_G \psi(x) \beta(x,y)|\eta_0|(x,y)\,d\Ha^2(x,y). \label{eq:mu-ident}
\end{align}
We claim that $|\eta_0|$ is the Jacobian of the projection $\pi_1:G\to\R^3_x$. As above, by \cite[Theorem 2.10]{AnST90} we can choose for $\Ha^2$-almost every $(x,y)\in G^*$ an embedded $C^1$ surface $\Sigma\subset\R^3$ and a $C^1$ map $\zeta: \Sigma\to S^2$ with \eqref{eq:repr-form}. 
In particular we then have
\begin{align*}
	\frac{1}{|\eta_0|}(x,y)=|\xi|(x,y)\,=\, |J(\Id\otimes \zeta)|(x)\,=\, \frac{1}{|J\pi_1|}(x,y), 
\end{align*}
see Lemma \ref{lem:xi}. Using \eqref{eq:mu-ident} and recalling the definition of $V_T$ from Section \ref{sec:CV-GG} we arrive at
\begin{align*}
	\mu(\psi) \,&=\, \int_G \psi(x) \beta(x,y)|\eta_0|(x,y)\,d\Ha^2(x,y) \,=\, \int_S \psi(x) \gamma(x)\,d\Ha^2(x) \,=\, \mu_{V_T}(\psi).
\end{align*}
\end{proof}
%
%
\begin{appendix}
%
\section{Origin of the energy $\Qfu_\e$ and results from \cite{LuPR14}}
Here we briefly describe the origin of the energy $\Qfu_\e$ considered in this paper and recall some properties from \cite{LuPR14}.

Biomembranes are formed by lipid molecules that self-assemble into thin bilayer structures. In \cite{PeRoe09} a mesoscale model was introduced that prescribes an energy for idealized and rescaled head and tail densities; such a model arises from a micro-scale description in which heads and tails are treated as separate particles. Evaluated in density functions $u,v$ of tail and head particles, respectively, the energy takes the form 
$$
\mathcal F_\e(u,v):=\left\{\begin{array}{ll}\displaystyle \int |\nabla u|+\frac{1}{\e}d_1(u,v), & \textrm{if $(u,v)\in\mathcal K_\e$}\\
+\infty, & \textrm{otherwise in $L^1(\R^n)\times L^1(\R^n)$,}
\end{array}\right.
$$
where $\int |\nabla u|$ is the total variation of $u$, $d_1(u,v)$ denotes the Monge-Kantorovich distance between $u$ and $v$ and for any $\e>0$ and fixed total mass $M_T>0$ we have set 
$$
	\K_\e:=\bigg\{(u,v)\in BV(\R^n;\{0,\e^{-1}\})\times  L^1(\R^n;\{0,\e^{-1}\}) : \int u =\int v=M_T,\,\textrm{$uv=0$ a.e.\,in $\R^n$}\bigg\}.
$$
It was shown in \cite{PeRoe09} that this energy favors structures where the $u$ mass is organized in thin layers of thickness $2\eps$, surrounded by two $v$ layers. More precisely, it was proved in \cite{PeRoe09} that the rescaled energy functional 
$$
\mathcal G_\e:=\frac{\mathcal F_\e -2M_T}{\e^2}
$$
in two space dimensions Gamma-converges to a generalized Euler elastica energy. In \cite{LuPR14} an analysis of the three-dimensional case has been started. In Theorem 2.1 of that paper a lower estimate for $F_\e(u,v)$ was proved. This estimate is given in terms of the boundary of the set $\{u>0\}$ and the ray directions $\theta$ associated to the Monge--Kantorovich mass transport problem and takes the following form: Consider a smooth connected compact surface $S$ that is part of the boundary of $\{u>0\}$, let $\nu$ denote the inner unit normal field of $S$ and consider the Lipschitz vector field $\theta=\nabla\phi:\R^3\to\R^3$, $|\theta|=1$ that describes the ray direction. Then
there exists a nonnegative measurable function $M\colon S \to \R$ such that $\theta \cdot \nu >0$ everywhere on $\{M>0\}$, and such that
\begin{equation}\label{estfond}
\mathcal G_\e(u,v) \geq  \frac{1}{\e^2}\int_S(M-1)^2\,d\mathcal  H^2+\frac{1}{\e^2}\int_S\bigg(\frac{1}{\theta \cdot \nu}-1\bigg)M^2\,d\mathcal  H^2+\int_S\frac{M^4}{(\theta \cdot \nu)^3}Q(D\theta)\,d\mathcal  H^2,
\end{equation}
with $Q$ as defined in \eqref{eq:def-Q}. Considering a sequence $(u_\eps,v_\e)$ with uniformly bounded energy we thus deduce that the mass distribution functions $M_\eps$ have to approach $1$, and that the ray directions have to become orthogonal to the boundary surfaces as $\eps$ tends to zero.

This estimate suggests that the Gamma-limit of $\mathcal G_\e$ (with respect of convergence of $u_\eps$ as measures) is given -- for limit measures given by a sufficiently regular surface $S$ equipped with unit density -- by
\begin{equation}\label{eq:def-G0-app}
	\mathcal G_{0}(S) =\, 2\int_{S}  \bigg(\frac{1}{4}H^2-\frac{1}{6}K\bigg)\, d\mathcal H^2
\end{equation}
where $H$ and $K$ are, respectively, the mean curvature and the Gauss curvature of $S$.

The corresponding upper estimate for $\mathcal G_\e$ has been proved in \cite[Theorem 1.5]{LuPR14}. The corresponding $\liminf$ estimate is much harder to obtain. The contribution of this paper is a major step in this direction: The functional $\Qfu_\e$ exactly corresponds to the right-hand side of \eqref{estfond}, when we restrict ourselves to constant mass $M_\eps\equiv 1$ and just one connected component $S_\e$ of the boundary $\partial\{u_\eps>0\}$. In this sense we have addressed here the deviation of the director field from the normal, but have neglected an additional deviation in the mass distribution on the surfaces. 

We finally restate a Lemma on the quadratic form $Q$ in \eqref{eq:def-Q} that we have used in the current paper. 

\begin{lemma}\label{lem:evDtheta}\cite[Lemma 3.6]{LuPR14}
For $\Ha^2$-almost all $p\in E$, $D\theta(p)$ is diagonalizable, and there exists a positively oriented orthonormal basis $\{v_1,v_2,\theta(p)\}$ of eigenvectors with $\det(v_1,v_2,\theta(p))=1$ and eigenvalues $\evla,\evla$ such that
\begin{align*}
	D\theta(p) v_1 = \evla v_1,\quad D\theta(p) v_2 = \evmu v_2, \quad D\theta(p)\theta(p) = 0.
\end{align*}
Moreover, we have
\begin{align}
	\tr D\theta \,&=\, \evla + \evmu,\qquad \tr \cof D\theta \,=\, \evla\evmu, \notag\\
	Q(D\theta) \,&=\, \frac{1}{4}(\evla+\evmu)^2-\frac{1}{6}\evla\evmu = \frac{1}{6}(\evla+\evmu)^2+\frac{1}{12}(\evla^2+\evmu^2). \label{eq:evDtheta}
\end{align}
\end{lemma}

%
%
\section{Exterior algebra and currents}\label{app:ext-algebra}
Denote by $\Lambda^k(\R^n),\,0\leq k\leq n$ the space of all {\it $k$-vectors} in $\R^n$, where we identify $\Lambda^1(\R^n)$ with $\R^n$ and set $\Lambda^0(\R^n):=\R$. If we denote by $\{e_1,\dots,e_n\}$ the standard basis of $\R^n$
then $\{e_i\wedge e_j\,:\, 1\leq i<j\leq n\}$ defines the standard basis of $\Lambda^2(\R^n)$. The euclidean scalar product of two $2$-vectors $v=\sum_{1\leq i<j\leq n}
\alpha_{ij}e_i\wedge e_j$, $w=\sum_{1\leq i<j\leq n}
\beta_{ij}e_i\wedge e_j$ is given by 
$$
	(v,w) \,:=\, \sum_{1\leq i<j\leq n} 
	\alpha_{ij}\beta_{ij}
$$ 
and the induced euclidean norm is denoted by $|\cdot|$.
We call $v$ a {\it simple $2$-vector} if $v$ can be written as $v=v_1\wedge v_2$. If in addition $v\neq 0$ the space $\spann(v_1,v_2)$ is called the {\it enveloping subspace}. We observe that
we have a one-to-one correspondence between simple two-vectors with unit norm and the Grassmann manifold of all oriented two-dimensional subspaces of $\R^n$. Another useful operation between vectors is the {\it interior multiplication} denoted by $\mres$ and defined as follows: if $v\in \Lambda^k(\R^n)$ and $w\in\Lambda^h(\R^n)$ with $k \ge h$ the vector $v\mres w$ belongs to $\Lambda^{k-h}(\R^n)$ and $\langle v\mres w,u\rangle=\langle v,w\wedge u\rangle$ holds true for any $u\in \Lambda^{k-h}(\R^n)$. We finally recall the definition of {\it Hodge operator}. We restrict to the case $p\in \{1,\dots,n-1\}$: $\ast\colon \Lambda^p(\R^n) \to \Lambda ^{n-p}(\R^n)$ is linear and defined starting from 
$$
\ast(e_1\wedge \cdots \wedge e_p):=\pm e_{p+1}\wedge \cdots \wedge e_n
$$
where we take ``$+$'' if the basis $\{e_1,\dots,e_n\}$ is positive oriented and ``$-$'' otherwise.

The space of all {\it $k$-covectors} in $\R^n$ is denoted by $\Lambda_k(\R^n)$. The standard dual basis of $\{e_1,\dots,e_n\}$ is denoted by $\{dx^1,\dots,dx^n\}$.
The euclidean scalar product of two $2$-covectors $\omega=\sum_{1\leq i<j\leq n}
\alpha_{ij}dx^i\wedge dx^j$, $\beta=\sum_{1\leq i<j\leq n}
\beta_{ij}dx^i\wedge dx^j$ is given by 
$$
	(\omega,\beta) \,:=\, \sum_{1\leq i<j\leq n} 
	\alpha_{ij}\beta_{ij}
$$ 
and the induced euclidean norm is denoted by $|\cdot|$.
The space $\Lambda_k(\R^n)$ is the dual of $\Lambda^k(\R^n)$. 

\subsection{Currents}
Let  $k \in \{0,\dots,n\}$. For $U\subseteq \R^n$ open, a {\it $k$-differential form} on $U$ is a map 
$$
\omega \colon U \to \Lambda_k(\R^n),\quad \omega(x)=\sum_{1\le i_1< \cdots <i_k\le n}\omega_{i_1\cdots i_k}(x)dx^{i_1}\wedge \cdots \wedge dx^{i_k}
$$
with $\omega_{i_1\cdots i_k}$ of class $C^\infty(U)$. We denote by $\mathcal D^k(U)$ the space of all $k$-differential forms with compact support in $U$, equipped with usual topology of distributions. We denote by $\|\omega\|:=\sup_{x\in U} |\omega(x)|$ the supremum norm.\\
For $\omega\in \mathcal D^k(V)$, $V\subset \R^m$ open, $\omega=\sum_{1\leq i_i<\dots<i_k\leq m} 
\omega_{i_1\dots i_k} dy^{i_1}\wedge\dots\wedge dy^{i_k}$ and $f \in C^\infty(U;V)$ we define the \emph{pullback} $f^\sharp \omega \in \mathcal D^k(U)$ by
$$
f^\sharp \omega(x):=\sum_{1\le i_1<\cdots<i_k\le m}
\omega_{i_1\cdots i_k}(f(x))df_x^{i_1}\wedge \cdots \wedge df_x^{i_k}, \quad \text{ for all } x\in U
$$
where for any scalar field $g$ on $\R^n$ the differential $dg_x$ is the 1-form defined by 
$$
dg_x:=\sum_k\frac{\partial g}{\partial x_k}dx^k.
$$

The space $\mathcal D_k(U)$ of {\it $k$-currents on $U$} is the dual of $\mathcal D^k(U)$. The  {\it boundary} $\partial T\in \mathcal D_{k-1}(U)$ of $T \in \mathcal D_k(U)$ is defined by  
$$
\langle \partial T,\omega\rangle:=\langle T,d\omega\rangle \quad \text{ for all } \omega \in \mathcal D^{k-1}(U).
$$
Let $W \subset U$ be open. The {\it mass} of $T\in \mathcal D_k(U)$ in $W$ is given by
$$
{\mathbb M}_W(T):=\sup\{\langle T,\omega\rangle : \omega \in \mathcal D^k(U),\,{\rm spt}\,\omega \subset\subset W,\,\|\omega\|\leq 1\}.
$$
We define the Radon measure $\mu_T$ on $U$ by
$
\mu_T(W):={\mathbb M}_W(T)
$
for $W\subset U$ open. The {\it support} of $T$, denoted by ${\rm spt}\,T$, is given by the support of the measure $\mu_T$. Let $T \in \mathcal D_k(U)$, $V\subset \R^m$ open, and $f\in C^\infty(U;V)$ {\it proper}, that is $f^{-1}(K)\cap {\rm spt}\,T$ is compact in $U$ whenever $K$ is compact in $V$. We then define $f_\sharp T \in  \mathcal D_k(V)$ by
$$
\langle f_\sharp T,\omega\rangle:=\langle T,\zeta f^\sharp \omega\rangle, \quad \text{ for all } \omega \in \mathcal D^k(V),
$$
where $\zeta$ is any function $C^\infty_c(U)$ such that $\zeta=1$ in a neighbourhood of ${\rm spt}\,T \cap {\rm spt}\,f^\sharp\omega$; we remark that $f_\sharp T$ does not depend on the choice of $\zeta$. 

\subsection{Auxiliary results}
For $x,y\in\R^3$ we have
\begin{align}
	(x\wedge y)_{ij} \,&=\, (x\times y)_k \eps_{ijk}\quad\text{ for all }1\leq i<j\leq 3. \label{eq:aux1}
\end{align}
\begin{lemma}\label{lem:LA1}
Let $A,B\in\R^{3\times 3}$, $y\in\R^3$ with $|y|=1$ and $B_{ij}:=\sum_k\eps_{ijk}y_k$. Then
\begin{align}
	y\cdot \cof A y \,&=\, \tr\cof BA, \label{eq:lemLA-1}\\
	y\cdot \cof A y \,&=\, \sum_k \tr\cof (\sum_r(B_{ir}A_{rj}y_k +B_{ir}A_{rk}y_j))_{ij}. \label{eq:lemLA-2}
\end{align}
If additionaly $Ay=0$ and $A^Ty=0$ then we have
\begin{align}
	y\cdot \cof A y \,&=\, \tr\cof A, \label{eq:lemLA-3}\\
	y\cdot \cof A y \,&=\, \sum_k \tr\cof (-A_{ij}y_k - A_{ik}y_j)_{ij}. \label{eq:lemLA-4}
\end{align}
\end{lemma}
\begin{proof}
Since $(\cof B)_{ij}=y_iy_j$ we have 
$$
\tr\cof BA=\sum_{i,k}(\cof A)_{ik}(\cof B)_{ki}=\sum_{i,k}(\cof A)_{ik}y_iy_k=y\cdot \cof A y
$$
which proves \eqref{eq:lemLA-1}. Formula \eqref{eq:lemLA-2} follows by direct computation. For any $k=1,2,3$ let $C^k\in\R^{3\times 3}$ with entries 
$$
C^k_{ij}:=\sum_rB_{ir}(A_{rj}y_k+A_{rk}y_j).
$$
Then 
$$
C^k=\left(\begin{array}{ccc}0&y_3&-y_2\\
-y_3&0&y_1\\
y_2&-y_1&0\end{array}\right)\left(\begin{array}{ccc}A_{11}y_k+A_{1k}y_1 & A_{12}y_k+A_{1k}y_2 & A_{13}y_k+A_{1k}y_3\\
A_{21}y_k+A_{2k}y_1 & A_{22}y_k+A_{2k}y_2 & A_{23}y_k+A_{2k}y_3\\
A_{31}y_k+A_{3k}y_1 & A_{32}y_k+A_{3k}y_2 & A_{33}y_k+A_{3k}y_3\end{array}\right).
$$
By straightforward calculations we obtain, using $y_1^2+y_2^2+y_3^2=1$, 
$$
\sum_k \tr\cof C^k=\sum_{i,j}(\cof A)_{ij}y_iy_j=y\cdot \cof A y.
$$
Now assume that $Ay=0$ and $A^Ty=0$. In order to prove \eqref{eq:lemLA-3} we need two general results  from the theory of matrices. First of all the equality
\begin{equation}\label{for-matrix}
\det(M+yy^T)=\det M+y\cdot \cof My
\end{equation}
holds true for any $M\in\R^{n\times n}$ and for any $y\in \R^n$ (see \cite[Prop.\,3.21]{Serr10}); moreover, the Cayley-Hamilton Theorem for a matrix $N\in\R^{3\times 3}$ says that 
\begin{equation}\label{for-matrixBIS}
\det N=\frac{(\tr N)^3-3\tr N \tr (N^2)+2\tr (N^3)}{6}.
\end{equation}
Notice that since $Ay=0$ and $y\ne 0$ we must have $\det A=0$, and therefore in order to prove \eqref{eq:lemLA-3}, taking into account \eqref{for-matrix}, it is sufficient to show that 
\begin{equation}\label{for-matrixTRIS}
\det(A+yy^T)=\tr\cof A.
\end{equation}
Since $Ay=A^Ty=0$ we easily get
$$
\tr ((A+yy^T)^2)=\tr (A^2)+\tr (yy^Tyy^T)=\tr(A^2)+1
$$
and 
$$
\tr ((A+yy^T)^3)=\tr (A^3)+\tr (yy^Tyy^Tyy^T)=\tr(A^3)+1
$$
Using \eqref{for-matrixBIS} we therefore obtain, since $\det A=0$, 
$$
\begin{aligned}
\det(A+yy^T)&=\frac{(\tr A+1)^3-3(\tr A+1) \tr ((A+yy^T)^2)+2\tr ((A+yy^T)^3)}{6}\\
&=\frac{(\tr A)^2-\tr (A^2)}{2}\\
&=(\cof A)_{11}+(\cof A)_{22}+(\cof A)_{33}
\end{aligned}
$$
which completes the proof of \eqref{for-matrixTRIS}. Finally, \eqref{eq:lemLA-4} follows by direct computation:  we easily have, for any $i=1,2,3$, using $y_1^2+y_2^2+y_3^2=1$ and $Ay=A^Ty=0$, 
$$
\sum_k\cof\left(\begin{array}{ccc}-A_{11}y_k-A_{1k}y_1 & -A_{12}y_k-A_{1k}y_2 & -A_{13}y_k-A_{1k}y_3\\
-A_{21}y_k-A_{2k}y_1 & -A_{22}y_k-A_{2k}y_2 & -A_{23}y_k-A_{2k}y_3\\
-A_{31}y_k-A_{3k}y_1 & -A_{32}y_k-A_{3k}y_2 & -A_{33}y_k-A_{3k}y_3\end{array}\right)_{ii}=(\cof A)_{ii}
$$
and then \eqref{eq:lemLA-4} follows.
\end{proof}

\end{appendix}
%

\end{document}